\documentclass[10pt]{amsart}
\usepackage{amsfonts,amssymb,amscd,amsmath,enumerate,verbatim,calc}
\usepackage{mathrsfs}
\numberwithin{equation}{section}
\newtheorem{theorem}{Theorem}[section]

\newtheorem{lemma}[theorem]{Lemma}
\newtheorem{proposition}[theorem]{Proposition}

\newtheorem{example}[theorem]{Example}
\newtheorem{remark}[theorem]{Remark}

\begin{document}

\title[Geometry of Tangent Bundles of Statistical Manifolds Equipped with Cheeger--Gromoll Type Metrics]
 { Geometry of Tangent Bundles of Statistical Manifolds Equipped with Cheeger--Gromoll Type Metrics}

\bibliographystyle{amsplain}

\author[Esmaeil Peyghan, Leila Nourmohammadifar and Ion Mihai]{E. Peyghan$^{1}$, L. Nourmohammadifar$^{1}$ and I. Mihai$^{2,*}$}
\address{$^{1}$Department of Mathematics, Faculty of Science, Arak University,
Arak, 38156-8-8349, Iran.}
\address{
$^{2}$Department of Mathematics, Faculty of Mathematics and Computer Science, University of
Bucharest, Bucharest 010014, Romania.}
\email{e-peyghan@araku.ac.ir, l-nourmohammadifar@araku.ac.ir, imihai@fmi.unibuc.ro}


\keywords{ Statistical manifold, Cheeger-Gromoll type metric,  statistical connection, skewness tensor, sectional curvature, scalar curvature. 
\\
*Corresponding author
\\
\text{2020 \ Mathematics Subject Classification. } 53B12, 53B20, 53B35, 62B10}

\begin{abstract}
	In this paper, we investigate the geometry of the tangent bundle $TM$ of a statistical manifold $(M,g,\nabla)$ endowed with a two-parameter family of generalized Cheeger--Gromoll metrics $g_{p,q}$.  We compute the associated the Levi--Civita connection $\nabla^{p,q}$ and express its curvature in terms of the Riemannian curvature and the skewness tensor $K$ of the base statistical manifold. 	
	We further analyze the behavior of geodesics, identify conditions under which the fibers of $TM$ are totally geodesic, and determine when the geodesic flow associated with $g_{p,q}$ is incompressible. Moreover, we establish necessary and sufficient conditions for the tangent bundle to admit constant sectional curvature. Several examples are provided to illustrate the theory, including statistically deformed Euclidean spaces and information geometric models such as the manifold of normal distributions. 
	The sectional curvature of $(TM, g_{p,q})$ is computed for horizontal, vertical, and mixed directions, leading to a concise expression for the corresponding scalar curvature. 
	\end{abstract}

\maketitle

\section{Introduction}

Statistical manifolds, introduced by S. Amari and H. Nagaoka in the late 20th century \cite{AMH}, are a key concept in information geometry, which connects differential geometry with statistics. In this framework, a family of probability distributions is viewed as a differentiable manifold equipped with a Riemannian metric and a pair of dual affine connections. The Riemannian metric, especially the Fisher information metric, measures how easily nearby distributions can be distinguished, while the affine connections capture statistical properties such as bias and efficiency. Statistical manifolds have many applications in statistics, machine learning, and theoretical physics, providing a geometric perspective for problems involving uncertainty and information \cite{AM, A, BG, BN}. In recent years, they have been studied extensively, yielding a number of important results \cite{Blaga, Cai, D, Furuhata, PNI, Peyghan, PNII, OP}.

For an open subset $\Theta$ of $\mathbb{R}^n$ and a sample space $\Omega$ with parameter $\theta = (\theta^1, \cdots, \theta^n)$, we call the set of probability density functions 
\[S = \{ p(x; \theta) : \int_{\Omega}p(x; \theta) = 1, \ \ p(x; \theta) >0, \ \ \theta \in \Theta \subseteq \mathbb{R}^n\},\]
as a statistical model. For a statistical model $S$, the semi-definite Fisher information matrix $g(\theta) = [g_{ij}(\theta)]$ is defined as
\begin{equation}\label{3}
g_{ij}(\theta) :=\int_{\Omega}\partial_{i}\ell_{\theta}\partial_{j}\ell_{\theta}p(x; \theta)\mathrm{d}x=E_p[\partial_{i}\ell_{\theta}\partial_{j}\ell_{\theta}],
\end{equation}
where $\ell_{\theta} = \ell(x; \theta) := \mathrm{log}p(x; \theta)$, $\partial_i := \frac{\partial}{\partial \theta^i}$, and $E_p[f]$ is the expectation of $f(x)$ with respect to $p(x; \theta)$. Equipping the space $S$ with such information matrices, it is called a statistical manifold in literature.

Historically, Fisher was the first to introduce the relation \eqref{3} as a mathematical formulation of information in 1920 (see \cite{F}). It is known that if the metric $g$ is positive-definite and all of its components converge to real numbers, then $(S, g)$ becomes a Riemannian manifold, and $g$ is called the Fisher metric on $S$. Using the Fisher metric $g$, an affine connection $\nabla$ associated with the probability density $p(x;\theta)$ is defined by
\begin{equation*}
\Gamma_{ij,k} = g(\nabla_{\partial_i}\partial_j, \partial_k) :=  E_p[(\partial_i\partial_j\ell_\theta)\partial_k\ell_\theta].
\end{equation*} 
A natural geometric object associated with any differentiable manifold is its tangent bundle, which carries rich geometric and topological structures. It provides the natural setting for defining vector fields and expressing fundamental geometric notions such as connections, curvature, and geodesics. Since the tangent bundle also serves as the phase space in dynamical systems and geometric mechanics, its geometry is central to understanding trajectories and developing higher-order geometric frameworks.

The study of tangent bundles, and more broadly of vector bundles, has a long and well-established history in differential geometry. In the 1960s and 1970s, significant progress was made in constructing and analyzing lifted metrics on tangent bundles, notably the Sasaki metric and the Cheeger--Gromoll metric \cite{CG1, S}. These lifted metrics extend the geometry of the base manifold to the tangent bundle, allowing a detailed investigation of geodesics, curvature, and global geometric properties in higher-dimensional settings. They also play an essential role in the geometry of phase spaces, in mechanics, and in various extensions of Riemannian geometry (see, for example, \cite{Abbassi, B, BL, PNII3}).

In the context of statistical manifolds, equipping the tangent bundle with a lifted metric allows us to explore second-order geometric features that are not visible on the base manifold alone \cite{Opozda}. Such metrics provide a natural way to study refined curvature properties, analyze geometric flows, and connect statistical structures to broader geometric theories, including Finsler geometry and geometric mechanics. Among these, Cheeger--Gromoll type metrics are particularly important; by introducing adjustable parameters, they extend the classical Cheeger--Gromoll construction and offer flexible tools for controlling curvature and the behavior of geodesics on the tangent bundle.

In this work, we extend Cheeger--Gromoll type metrics to the tangent bundles of statistical manifolds by introducing a two-parameter family of metrics $g_{p,q}$. We derive explicit formulas for the Levi--Civita connection and the curvature, and identify conditions under which the fibers are totally geodesic, the geodesic flow is incompressible, and the tangent bundle has constant sectional curvature. This study provides a natural geometric framework for exploring higher-order curvature properties in information geometry and offers tools for further investigations in geometric statistics and related areas.

The organization of the paper is as follows. 
In Section~2, we recall basic definitions and fundamental properties of statistical manifolds. 
Section~3 introduces the generalized Cheeger--Gromoll type metric $g_{p,q}$ on the tangent bundle $TM$. 
In Section~4, we study the Levi--Civita connection $\nabla^{p,q}$ of $(TM, g_{p,q})$, deriving explicit formulas using horizontal and vertical lifts. We also analyze the behavior of geodesics, characterize the conditions under which the fibers are totally geodesic, and determine when the geodesic flow is incompressible. In addition, we compute the curvature tensor of $\nabla^{p,q}$ and examine its dependence on the geometric data of the base statistical manifold $(M,g,\nabla)$. 
Section~5 derives conditions under which the tangent bundle admits constant sectional curvature. Several examples are presented to illustrate the theory, including Euclidean spaces with statistical deformations and information-geometric models such as the manifold of normal distributions. 
Finally, in Section~6, we investigate the scalar curvature of $(TM, g_{p,q})$ equipped with the generalized Cheeger--Gromoll metric.

\section{Statistical Manifolds}

Let $M$ be an $n$-dimensional differentiable manifold, and let $(U, x^i)$, $i = 1, \ldots, n$, be a local chart around a point $x \in U \subset M$. With respect to the coordinate system $(x^i)$, the local frame $\left. \frac{\partial}{\partial x^i} \right|_x$ serves as a basis for the tangent space $T_xM$.
Let $\nabla$ be an affine connection on $M$, and let $g$ be a pseudo-Riemannian metric on $M$. Then the connection $\nabla^*$ defined by
	\begin{align*} 
	Xg(Y,Z) = g(\nabla_X Y, Z) + g(Y, \nabla^*_X Z),
	\end{align*}
	is called the {dual (or conjugate) connection} of $\nabla$ with respect to the metric $g$.
The connection $\nabla$ is called a {Codazzi connection} if the cubic tensor field $\mathcal{C} = \nabla g$ is totally symmetric; that is, it satisfies the Codazzi equations:
	\begin{align*} 
	(\nabla_X g)(Y, Z) = (\nabla_Y g)(X, Z) = (\nabla_Z g)(X, Y), \quad \forall X, Y, Z \in \mathfrak{X}(M),
	\end{align*}
	where
	\begin{align} \label{M2'}
	(\nabla_X g)(Y, Z) = X g(Y, Z) - g(\nabla_X Y, Z) - g(Y, \nabla_X Z).
	\end{align}
	In local coordinates, the components of the cubic tensor $\mathcal{C}$ are given by
	\begin{align*} 
	\mathcal{C}_{ijk} = \partial_k g_{ij} - \Gamma^h_{ik} g_{jh} - \Gamma^h_{jk} g_{ih}, \qquad \mathcal{C}_{ijk} = \mathcal{C}_{jik} = \mathcal{C}_{kij},
	\end{align*}
	where $\partial_i := \frac{\partial}{\partial x^i}$ and $\Gamma^i_{jk}$ are the Christoffel symbols of the connection $\nabla$.
	 A triple $(M, g, \nabla)$ is called a {statistical manifold} if $\nabla$ is a torsion-free Codazzi connection. In particular, when the cubic tensor field $\mathcal{C}$ vanishes, the connection $\nabla$ reduces to the Levi-Civita connection $\widehat{\nabla}$ of $g$.
It is worth noting that if $(M, g, \nabla)$ is a statistical manifold, then its dual $(M, g, \nabla^*)$ is also a statistical manifold. Moreover, their cubic tensors satisfy the relation
\[
\mathcal{C}^*(X, Y, Z) := (\nabla^*_X g)(Y, Z) = -\mathcal{C}(X, Y, Z), \quad \forall X, Y, Z \in \mathfrak{X}(M).
\]

For a statistical structure $(g, \nabla)$, we define a $(1,2)$-tensor field $K: \mathfrak{X}(M) \times \mathfrak{X}(M) \rightarrow \mathfrak{X}(M)$ by
\begin{align} \label{L33}
K_X Y = \nabla_X Y - \widehat{\nabla}_X Y, \quad \forall X, Y \in \mathfrak{X}(M),
\end{align}
where $\widehat{\nabla}$ is the Levi-Civita connection of $g$. Then, the tensor $K$ satisfies the following symmetry properties:
\begin{align} \label{def1}
K_X Y = K_Y X, \qquad g(K_X Y, Z) = g(Y, K_X Z).
\end{align}
Conversely, if a Riemannian metric $g$ and a $(1,2)$-tensor $K$ satisfy \eqref{def1}, the affine connection defined by $\nabla := \widehat{\nabla} + K$ becomes a statistical connection.
Moreover, the skewness tensor is related to the dual connection $\nabla^*$ by
\[
K = \widehat{\nabla} - \nabla^* = \frac{1}{2}(\nabla - \nabla^*),
\]
and the cubic tensor $\mathcal{C}$ satisfies
\begin{align} \label{M4}
\mathcal{C}(X, Y, Z) = (\nabla_X g)(Y, Z) = -2g(K_X Y, Z), \quad \forall X, Y, Z \in \mathfrak{X}(M).
\end{align}
Furthermore, the Levi-Civita derivative of the cubic tensor is given by
\begin{align*} 
(\widehat{\nabla}_X \mathcal{C})(Y, Z, W) = -2g\big( (\widehat{\nabla}_X K)(Y, Z), W \big).
\end{align*}
Let $(M, g, \nabla)$ be a statistical manifold. The curvature tensor associated with the affine connection $\nabla$ is defined by
\begin{align} \label{2N}
R^\nabla(X, Y)Z := \nabla_X \nabla_Y Z - \nabla_Y \nabla_X Z - \nabla_{[X, Y]} Z,
\end{align}
for all vector fields $X, Y, Z \in \mathfrak{X}(M)$.
For simplicity, we denote the curvature tensors of the connections $\nabla$, $\nabla^*$, and the Levi-Civita connection $\widehat{\nabla}$ by $R$, $R^*$ and $\widehat{R}$, respectively.
The curvature tensor $R$ satisfies the following identities   \cite{CU}:
\begin{align*}
&R(X, Y, Z, W)  = -R(Y, X, Z, W), \\
&R(X, Y, Z, W) = -R^*(X, Y, W, Z), \\
&R(X, Y, Z, W) + R(Y, Z, X, W) + R(Z, X, Y, W) = 0, 
\end{align*}
where $R(X, Y, Z, W) = g(R(X, Y)Z, W)$.
Moreover, the curvature tensor $R$ can be expressed in terms of the Levi-Civita curvature and the skewness tensor $K$ as follows:
\begin{align*} 
R(X, Y)Z &= \widehat{R}(X, Y)Z + (\widehat{\nabla}_X K)(Y, Z) - (\widehat{\nabla}_Y K)(X, Z) + [K_X, K_Y]Z \\
&= \widehat{R}(X, Y)Z + (\nabla_X K)(Y, Z) - (\nabla_Y K)(X, Z) - [K_X, K_Y]Z. \nonumber
\end{align*}
 \section{Generalized Cheeger-Gromoll Type Metrics on Statistical Manifolds}
 
 Let $\nabla$ be a statistical connection on a manifold $M$, and let $TM$ denote its tangent bundle with the natural projection $\pi: TM \rightarrow M$. The differential of the projection, $\pi_*: TTM \rightarrow TM$, allows us to split the tangent bundle of $TM$ into horizontal and vertical subbundles:
 \begin{align*}
 TTM = HTM \oplus VTM,
 \end{align*}
 where the vertical distribution is $VTM := \ker \pi_*$ and $HTM$ is the complementary horizontal distribution defined by $\nabla$. 
 
 Given a local coordinate system $(U, x^i)$ on $M$, we get an induced coordinate system $(\pi^{-1}(U), x^i, u^i)$ on $TM$, where $x^i$ are the base coordinates and $u^i$ are fiber (directional) coordinates. Let $\Gamma^k_{ij}$ be the Christoffel symbols of $\nabla$. Then the adapted frame on $TM$ is
 \begin{align}\label{Mo3}
 \delta_i := \frac{\delta}{\delta x^i} = \frac{\partial}{\partial x^i} - u^r \Gamma_{ir}^k \frac{\partial}{\partial u^k}, \quad
 \partial_{\bar{i}} := \frac{\partial}{\partial u^i},
 \end{align}
 which span $H_u TM$ and $V_u TM$, respectively. The dual coframe is $\{dx^i, \delta u^i\}$, where
 \[
 \delta u^i := du^i + u^r \Gamma^i_{rk} dx^k.
 \]
 For any $X \in \mathfrak{X}(M)$, we denote by $X^H$ and $X^V$ its horizontal and vertical lifts to $TM$.
 
 The Lie brackets of the adapted frame satisfy
 \begin{align} \label{L5}
 [X^H, Y^H] = [X, Y]^H - (R(X, Y)u)^V, \quad 
 [X^H, Y^V] = (\nabla_X Y)^V, \quad 
 [X^V, Y^V] = 0,
 \end{align}
 for all $X, Y \in \mathfrak{X}(M)$.
Define the \emph{energy density} of a vector $u$ at $x$ by
 \begin{align*}
 \tau := g_x(u, u).
 \end{align*}
 
 If $(M, g, \nabla)$ is a statistical manifold, the following differential relations hold for any $X, Y \in \mathfrak{X}(M)$:
 \begin{align} \label{Nour1}
 \begin{cases}
 X^H(g(u, u) \circ \pi) = -2g(K_X u, u) \circ \pi, \\
 X^H(g(Y, u) \circ \pi) = (g(\widehat{\nabla}_X Y, u) - g(Y, K_X u)) \circ \pi, \\
 X^V(g(u, u) \circ \pi) = 2g(X, u) \circ \pi, \\
 X^V(g(Y, u) \circ \pi) = g(X, Y) \circ \pi.
 \end{cases}
 \end{align}
 
 Following \cite{BL}, we define a two-parameter family of Riemannian metrics $g_{p,q}$ on the tangent bundle $TM$, called \emph{generalized Cheeger--Gromoll type metrics}, by
  \begin{align} \label{L20}
 \begin{cases}
 g_{p,q(x,u)}(X^H, Y^H) = g_x(X, Y), \\
 g_{p,q(x,u)}(X^H, Y^V) = 0, \\
 g_{p,q(x,u)}(X^V, Y^V) = \dfrac{1}{\alpha^p} \Big( g_x(X, Y) + q \, g_x(X, u) g_x(Y, u) \Big),
 \end{cases}
 \end{align}
 where $\alpha := 1 + \tau$. In particular, $p = q = 0$ gives the \emph{Sasaki metric} ${}^{S}g$, while $p = q = 1$ gives the classical \emph{Cheeger-Gromoll metric} ${}^{CG}g$.
Define the open and boundary domains associated with $q$:
 \[
 BM_q := \{ (x, u) \in TM \mid q \tau > -1 \}, \qquad
 SM_q := \{ (x, u) \in TM \mid q \tau = -1 \}.
 \]
 If $q \geq 0$, then $BM_q = TM$ and $SM_q = \emptyset$. If $q < 0$, then $BM_q$ forms a Riemannian ball bundle of $g_{p,q}$ and $SM_q$ is its sphere bundle boundary.
 
 We next compute some fundamental properties of the generalized Cheeger-Gromoll metric on the tangent bundle:
 \begin{proposition}
 	On  a statistical manifold $(M,g, \nabla)$ with 
  the tangent bundle $TM$ equipped with the Cheeger-Gromoll type metric $g_{p,q}$, we have 
  	\begin{align*}
 	g_{p,q}(X^V,u^V)=&	\frac{1+q\tau}{\alpha^p}g(X,u), \\
 	X^Hg_{p,q}(Y^V,Z^V)= &\frac{2p}{\alpha}g(K_Xu,u)g_{p,q}(Y^V,Z^V)-2g_{p,q}((K_XY)^V,Z^V)+g_{p,q}((\nabla_XY)^V,Z^V) \\
 	&\quad +g_{p,q}(Y^V,( \nabla_XZ)^V)-\frac{2q}{\alpha^p}g(Y,u)g(K_XZ,u),\nonumber\\
 	X^Vg_{p,q}(Y^V,Z^V)=&-\frac{2p}{\alpha}g(X,u)g_{p,q}(Y^V,Z^V)+\frac{q}{\alpha^p}\{g(Y,X)g(Z,u)+g(Z,X)g(Y,u)\}.
 	\end{align*}
 	 \end{proposition}
\begin{proof}
		This follows directly from \eqref{Nour1} and \eqref{L20}.
	\end{proof}
\section{The Levi-Civita Connection and the curvature tensor of $(TM, g_{p,q})$}
In this section, we study the Levi-Civita connection $\nabla^{p,q}$ of the tangent bundle $TM$ equipped with the Cheeger-Gromoll type metric $g_{p,q}$, which is induced by a statistical manifold $(M,g,\nabla)$. We provide explicit formulas for $\nabla^{p,q}$ in terms of horizontal and vertical lifts and examine the influence of the properties of the statistical connection $\nabla$ on the geometry of $TM$.
 In particular, we analyze the behavior of geodesics, the conditions for totally geodesic fibers, and the incompressibility of the geodesic flow. These results highlight the rich geometric structure arising from Cheeger-Gromoll type metrics on tangent bundles. 	In the following, we compute the curvature tensor of the Levi--Civita connection $\nabla^{p,q}$ for the Cheeger--Gromoll type metric $g_{p,q}$ on $TM$. 

We compute the Levi-Civita connection $\nabla^{p,q}$ of $(TM,g_{p,q})$ induced by the statistical manifold $(M,g,\nabla)$. Its explicit expression is obtained from the classical formula
\begin{align}\label{L34}
2g_{p,q}(\nabla^{p,q}_{\bar{X}}\bar{Y},\bar{Z}) &= \bar{X}\, g_{p,q}(\bar{Y},\bar{Z}) + \bar{Y}\, g_{p,q}(\bar{X},\bar{Z}) - \bar{Z}\, g_{p,q}(\bar{X},\bar{Y}) \\
&\quad + g_{p,q}(\bar{Z},[\bar{X},\bar{Y}]) - g_{p,q}(\bar{Y},[\bar{X},\bar{Z}]) - g_{p,q}(\bar{X},[\bar{Y},\bar{Z}]), \nonumber
\end{align}
for any vector fields $\bar{X},\bar{Y},\bar{Z} \in \chi(TM)$.
\begin{remark} \label{LM6'}
	Let $(M, g, \nabla)$ be a statistical manifold. Define the $(1,3)$-tensor fields $\widetilde{R}$ and $\widetilde{R}^*$ by
	\begin{align} \label{M3}
	g(\widetilde{R}(X, Y)Z, W) := R(Z, W, X, Y), \quad g(\widetilde{R}^*(X, Y)Z, W) := R^*(Z, W, X, Y),
	\end{align}
	for all $X, Y, Z, W \in \mathfrak{X}(M)$. Then, the following symmetries hold:
	\begin{align*}
	g(\widetilde{R}(X, Y)Z, W) = -g(\widetilde{R}(X, Y)W, Z) = -g(\widetilde{R}^*(Y, X)Z, W).
	\end{align*}
\end{remark}
\begin{theorem}\label{L2}
	Let $(M,g, \nabla)$ be a statistical manifold
	and ${{\nabla}^{p,q}}$  be the Levi-Civita connection of the tangent bundle $TM$ equipped with the Cheeger-Gromoll type metric $g_{p,q}$. Then ${{\nabla}^{p,q}}$ at $(x,u)\in BM_q$ satisfies
	\begin{align*}
&(i)\ \	{{\nabla}}^{p,q}_{X^H}Y^H=(\nabla_XY)^H-(K_XY)^H-\frac{1}{2}(R(X,Y)u)^V,	\\
&(ii)\	{{\nabla}}^{p,q}_{X^H}Y^V=\frac{1}{2\alpha^p}\{(\widetilde R(u,Y)X)^H+qg(Y,u)(\widetilde R(u,u)X)^H\}+\frac{p}{\alpha}g(K(X,u),u)Y^V-(K(X,Y))^V\\
&\ \ \ \ \ \ \ \ \ \ \ \ \ \ \ \  +(\nabla_XY)^V-qg(Y,u)\{(K(X,u))^V-\frac{q}{1+q\tau}g(K(X,u),u)u^V\},\\
&(iii)\	{{\nabla}}^{p,q}_{X^V}Y^H=\frac{1}{2\alpha^p}\{(\widetilde R(u,X)Y)^H+qg(X,u)(\widetilde R(u,u)Y)^H\}+\frac{p}{\alpha}g(K(Y,u),u)X^V-(K(X,Y))^V\\
&\ \ \ \ \ \ \ \ \ \ \ \ \ \ \ \  -qg(X,u)\{(K(Y,u))^V-\frac{q}{1+q\tau}g(K(Y,u),u)u^V\},\\
&
(iv)\	{{\nabla}}^{p,q}_{X^V}Y^V=-\frac{p}{\alpha}
(K_uu)^Hg_{p,q}(Y^V,X^V)+\frac{1}{\alpha^p}(K_XY)^H+\frac{q}{\alpha^p}(K_Yu)^Hg(X,u)+\frac{q}{\alpha^p}(K_Xu)^Hg(Y,u)\\
&
\ \ \ \ \ \ \ \ \ \ \ \ \ \ \ \ \ \ \ \ \ \ \ \ \  -\frac{p}{\alpha}(g(X,u)Y^V+g(Y,u)X^V)+\mathcal{M}g(X,Y)u^V  +\mathcal{N}g(X,u)g(Y,u)u^V,
	\end{align*}
	where $\mathcal{M}=\frac{q\alpha+p}{\alpha(1+q\tau)}$ and $\mathcal{N}=\frac{pq}{\alpha(1+q\tau)}$.
	\end{theorem}
\begin{proof}
	Using  \eqref{M2'}, \eqref{L5} and \eqref{L20} in \eqref{L34}, we first obtain the horizontal part as follows:
	\[
	2g_{p,q}(\nabla^{p,q}_{X^H} Y^H, Z^H) = \mathcal{C}(X,Y,Z) + 2g(\nabla_X Y, Z) = 2g(-K_X Y + \nabla_X Y, Z),
	\]
	which can be rewritten as
	\[
g_{p,q}(\nabla^{p,q}_{X^H} Y^H, Z^H) = 	g_{p,q}\big((-K_X Y)^H + (\nabla_X Y)^H, Z^H \big).
	\]
	For the vertical part, we have
	\[
	2g_{p,q}(\nabla^{p,q}_{X^H} Y^H, Z^V) = 2g_{p,q}\big(Z^V, (R(X,Y) u)^V \big).
	\]
	This completes the proof of part (i).
	For the horizontal part of (ii), using \eqref{M3} we get
	\[
	2g_{p,q}(\nabla^{p,q}_{X^H} Y^V, Z^H) = -\frac{1}{\alpha^p} \big\{ g(Y, R(Z,X) u) + q g(Y,u) g(R(Z,X) u, u) \big\},
	\]
	which can be expressed in terms of the metric \(g_{p,q}\) as
	\[
2g_{p,q}(\nabla^{p,q}_{X^H} Y^V, Z^H) =	\frac{1}{\alpha^p} g_{p,q} \big( (\widetilde{R}(u,Y) X)^H + q g(Y,u) (\widetilde{R}(u,u) X)^H, Z^H \big).
	\]
For the vertical part we have
	\[
	2g_{p,q}(\nabla^{p,q}_{X^H} Y^V, Z^V) = g_{p,q} \Big( \frac{2p}{\alpha} g(K(X,u), u) Y^V - 2 (K(X,Y))^V + 2 (\nabla_X Y)^V, Z^V \Big) - \frac{2q}{\alpha^p} g(Y,u) g(K(X,u), Z).
	\]
	Since the horizontal and vertical distributions are orthogonal with respect to \(g_{p,q}\), the first identity gives the horizontal part of \(\nabla^{p,q}_{X^H} Y^V\) and the second identity gives the vertical part; that is,
	\[
	\nabla^{p,q}_{X^H} Y^V = W_1^H + W_2^V,
	\]
	whit
	\[
	W_1 = \frac{1}{2 \alpha^p} \big\{ \widetilde{R}(u,Y) X + q g(Y,u) \widetilde{R}(u,u) X \big\},
	\]
	and
	\[
	W_2 = \frac{p}{\alpha} g(K(X,u), u) Y - K(X,Y) + \nabla_X Y + W,
	\]
where
	\begin{align}\label{M1}
	g_{p,q}(W^V, Z^V) = - \frac{q}{\alpha^p} g(Y,u) g(K(X,u), Z).
	\end{align}
		By substituting \(Z = u\) in the above equation, we obtain
	\[
	g_{p,q}(W^V, u^V) = - \frac{q}{\alpha^p} g(Y,u) g(K(X,u), u),
	\]
	which implies
	\[
	\frac{1 + q \tau}{\alpha^p} g(W, u) = - \frac{q}{\alpha^p} g(Y,u) g(K(X,u), u).
	\]
	Thus, it follows
	\[
	g(W,u) = - \frac{q}{1 + q \tau} g(Y,u) g(K(X,u), u).
	\]
		On the other hand, from \eqref{M1} we have
	\[
	\frac{1}{\alpha^p} \big\{ g(W,Z) + q g(W,u) g(Z,u) \big\} = - \frac{q}{\alpha^p} g(Y,u) g(K(X,u), Z).
	\]
	Substituting the value of \(g(W,u)\) into the above equation yields
	\[
	g(W,Z) = \frac{q^2}{1 + q \tau} g(Y,u) g(K(X,u), u) g(u,Z) - q g(Y,u) g(K(X,u), Z).
	\]
	Therefore, we have
	\[
	W = - q g(Y,u) \Big( K(X,u) - \frac{q}{1 + q \tau} g(K(X,u), u) u \Big),
	\]
this completes the proof of (ii). According to (\ref{L5}) and (ii), we have (iii).
			For (iv), in the same way, it follows
	\[
	2 g_{p,q}(\nabla^{p,q}_{X^V} Y^V, Z^H) = - Z^H g_{p,q}(X^V, Y^V) - g_{p,q}(X^V, [Y^V, Z^H]) - g_{p,q}(Y^V, [X^V, Z^H]).
	\]
Using	(\ref{L5}) and (\ref{L20}), the last equation implies
	\[
	2 g_{p,q}(\nabla^{p,q}_{X^V} Y^V, Z^H) =	\frac{p}{\alpha} \mathcal{C}(Z,u,u) g_{p,q}(Y^V, X^V) - \frac{1}{\alpha^p} \mathcal{C}(X,Y,Z) - \frac{q}{\alpha^p} \mathcal{C}(Z,Y,u) g(X,u) - \frac{q}{\alpha^p} g(Y,u) \mathcal{C}(X,Z,u).
	\]
Applying \eqref{M4}, we obtain the horizontal part of (iv).
					Similarly, for the vertical part we have
	\[
	2 g_{p,q}(\nabla^{p,q}_{X^V} Y^V, Z^V) = X^V g_{p,q}(Y^V, Z^V) + Y^V g_{p,q}(X^V, Z^V) - Z^V g_{p,q}(X^V, Y^V),
	\]
	which gives
	\begin{align*}
	2 g_{p,q}(\nabla^{p,q}_{X^V} Y^V, Z^V) = &2 g_{p,q} \Big( - \frac{p}{\alpha} ( g(X,u) Y^V + g(Y,u) X^V ) + \frac{q \alpha + p}{\alpha (1 + q \tau)} g(X,Y) u^V \\
	&+ \frac{p q}{\alpha (1 + q \tau)} g(X,u) g(Y,u) u^V, Z^V \Big).
	\end{align*}
	This completes the proof.
\end{proof}
\begin{lemma}
	In a statistical manifold \((M,g,\nabla)\), for any \(X \in \chi(M)\), we have:
	\begin{align*}
	\nabla^{p,q}_{X^H} u^V =\; & \frac{1 + q \tau}{2\alpha^p} (\widetilde{R}(u,u)X)^H + \frac{p}{\alpha} g(K_X u, u) u^V - (1 + q \tau)(K_X u)^V \hspace{1.3cm}\\
	& + \frac{q^2 \tau}{1 + q \tau} g(K_X u, u) u^V,
	\end{align*}
	\begin{align*}
	\nabla^{p,q}_{X^V} u^V =\; & \frac{-p(1 + q\tau) + q\alpha}{\alpha^{p+1}} g(X,u)(K_u u)^H + \frac{1 + q\tau}{\alpha^p} (K_X u)^H + \left(1 - \frac{p \tau}{\alpha} \right) X^V \\
	& + \left( \mathcal{M} + \tau \mathcal{N} - \frac{p}{\alpha} \right) g(X,u) u^V.
	\end{align*}
\end{lemma}
\begin{proof}
Since
	\begin{align}\label{Mo4}
	\nabla^{p,q}_{X^H} u^V 
	&= \nabla^{p,q}_{X^H} \left( u^i \frac{\partial}{\partial u^i} \right) 
	= X^H(u^i) \frac{\partial}{\partial u^i} + u^i \nabla^{p,q}_{X^H} \left( \frac{\partial}{\partial u^i} \right),
	\end{align}
	thus using part (ii) of Theorem \ref{L2}, it follows
	\begin{align}\label{Mo5}
	u^i \nabla^{p,q}_{X^H} \left( \frac{\partial}{\partial u^i} \right) 
	=&\; u^i \Bigg[ \frac{1}{2\alpha^p} \Big( (\widetilde{R}(u, \tfrac{\partial}{\partial x^i})X)^H + q g\left( \tfrac{\partial}{\partial x^i}, u \right)(\widetilde{R}(u,u)X)^H \Big) \\
	&\quad + \frac{p}{\alpha} g(K(X,u), u) \left( \frac{\partial}{\partial x^i} \right)^V - (K(X, \tfrac{\partial}{\partial x^i}))^V \nonumber\\
	&\quad + (\nabla_X \tfrac{\partial}{\partial x^i})^V - q g\left( \tfrac{\partial}{\partial x^i}, u \right) \left( (K(X,u))^V - \frac{q}{1 + q\tau} g(K(X,u), u) u^V \right) \Bigg].\nonumber
	\end{align}
	On the other hand, from (\ref{Mo3}) we have
	\begin{align}\label{Mo6}
	X^H(u^i) \frac{\partial}{\partial u^i}=-u^i (\nabla_X \tfrac{\partial}{\partial x^i})^V.
	\end{align}
Substituting \eqref{Mo5} and \eqref{Mo6} in \eqref{Mo4}, and simplifying, we obtain the first identity.
		Similarly, applying part (iv) of Theorem \ref{L2}, we find \(\nabla^{p,q}_{X^V} u^V\) directly and obtain the second identity.
	\end{proof}
\begin{example}\label{M111}
	We consider the exponential distribution with parameter \(\xi\) given by
	\[
	p(x; \xi) = \xi e^{-\xi x}, \quad \xi \in (0, \infty), \quad x \in [0, \infty).
	\]
	The associated statistical model \( M = \{ p(x; \xi) \} \) is one-dimensional. The Fisher information metric is computed as
	\[
	g_{11} = g(\frac{\partial}{\partial \xi}, \frac{\partial}{\partial \xi}) = \frac{1}{\xi^2}.
	\]
	Also, the Christoffel symbols \(\widehat{\Gamma}^1_{11}\) of the Levi-Civita connection \(\widehat{\nabla}\) and the components \(K^1_{11}\) of the tensor field \(K\) on \(M\) are given by
	\[
	\widehat{\Gamma}^1_{11} = -\frac{1}{\xi}, \quad K^1_{11} = \frac{1}{\xi}.
	\]
	Hence, \((M, g, \nabla = \widehat{\nabla} + K)\) forms a statistical manifold with cubic tensor \(\mathcal{C}_{111} = -\frac{2}{\xi^3}\).
		The Cheeger-Gromoll type metric \(g_{p,q}\) on the tangent bundle \(TM\) is given by
	\[
	g_{p,q}(\delta_1, \delta_1) = \frac{1}{\xi^2}, \quad
	g_{p,q}(\delta_1, \partial_{\bar{1}}) = 0, \quad
	g_{p,q}(\partial_{\bar{1}}, \partial_{\bar{1}}) = \frac{1}{\alpha^p} \{ \frac{1}{\xi^2} + q u_1 u_1 \},
	\]
	where \(u_1 = \frac{u^1}{\xi^2}\).
		Using Theorem \ref{L2}, the connection coefficients are computed as
	\begin{align*}
	\nabla^{p,q}_{\delta_1} \delta_1 &= -\frac{1}{\xi} \delta_1, \\
	\nabla^{p,q}_{\delta_1} \partial_{\bar{1}} &= \frac{q (p-2) (y^1)^4 - \xi^2 (y^1)^2 (1 + 2q - p) - \xi^4}{\xi^5 \alpha (1 + q \tau)} \partial_{\bar{1}} = \nabla^{p,q}_{\partial_{\bar{1}}} \delta_1, \\
	\nabla^{p,q}_{\partial_{\bar{1}}} \partial_{\bar{1}} &= \frac{-p q (y^1)^4 + \xi^2 (y^1)^2 (2 q \alpha - p) + \alpha \xi^4}{\xi^5 \alpha^{p+1}} \delta_1 - \frac{y^1 (\xi^2 (p - q) + q (y^1)^2 (p - 1))}{\xi^4 \alpha (1 + q \tau)} \partial_{\bar{1}}.
	\end{align*}
		In particular, for the cases \(p = q = 0\) and \(p = q = 1\), we have
	\[
	\nabla^{0,0}_{\delta_1} \delta_1 = -\frac{1}{\xi} \delta_1, \quad
	\nabla^{0,0}_{\delta_1} \partial_{\bar{1}} = -\frac{1}{\xi} \partial_{\bar{1}} = \nabla^{0,0}_{\partial_{\bar{1}}} \delta_1, \quad
	\nabla^{0,0}_{\partial_{\bar{1}}} \partial_{\bar{1}} = \frac{1}{\xi} \delta_1,
	\]
	and
	\[
	\nabla^{1,1}_{\delta_1} \delta_1 = -\frac{1}{\xi} \delta_1, \quad
	\nabla^{1,1}_{\delta_1} \partial_{\bar{1}} = -\frac{1}{\xi} \partial_{\bar{1}} = \nabla^{1,1}_{\partial_{\bar{1}}} \delta_1, \quad
	\nabla^{1,1}_{\partial_{\bar{1}}} \partial_{\bar{1}} = \frac{1}{\xi} \delta_1.
	\]
\end{example}

 \begin{proposition}
 	Let $(M, g, \nabla)$ be a statistical manifold and let its tangent bundle $TM$ be equipped with a Cheeger--Gromoll type metric $g_{p,q}$. Then the fibers of $TM$ are totally geodesic if and only if one of the following conditions holds:
 	\begin{enumerate}[i)]
 		\item $p = q = 0$, i.e., $g_{p,q}$ reduces to the Sasaki metric;
 		\item $K = 0$, i.e., $\nabla$ coincides with the Levi-Civita connection $\widehat{\nabla}$ of $(M,g)$.
 	\end{enumerate}
 \end{proposition}
 
 \begin{proof}
 	Using Theorem~\ref{L2}, the fibers of $TM$ are totally geodesic if and only if
 	\begin{equation}\label{eq:tot_geod}
 	-\frac{p}{\alpha^{p+1}} \big\{ g(X,Y) + q\, g(X,u) g(Y,u) \big\} K_u u
 	+ \frac{q}{\alpha^p} \big\{ g(X,u) K_Y u + g(Y,u) K_X u \big\} = 0,
 	\end{equation}
 	for all $u, X, Y \in T_x M$ and all $x \in M$.
 	 	Choose $0 \neq X = Y \perp u$. Substituting into \eqref{eq:tot_geod} yields
 	\begin{equation}\label{Mo7}
 	\frac{p}{\alpha^{p+1}} K_u u = 0, \qquad \forall\, u \in TM.
 	\end{equation}
 	Then \eqref{eq:tot_geod} simplifies to
 	\[
 	\frac{q}{\alpha^p} \big\{ g(X,u) K_Y u + g(Y,u) K_X u \big\} = 0,
 	\qquad \forall\, x \in M,\ u,X,Y \in T_xM.
 	\]
 	 	Now take $0 \neq X = u \perp Y$. This gives
 	\begin{equation}\label{Mo8}
 	\frac{q}{\alpha^p} K_Y u = 0,
 	\qquad \forall\, x \in M,\ u \perp Y.
 	\end{equation}
 	 	If $p \neq 0$, then \eqref{Mo7} implies $K_u u = 0$, and \eqref{Mo8} becomes $q K_Y u = 0$. Thus either $q=0$ or $K=0$.
 	 	If $p = 0$, then from \eqref{Mo8} we obtain either $q=0$ or
 	\begin{equation}\label{Mo9}
 	K_Y u = 0,
 	\qquad \forall\, x \in M,\ u \perp Y.
 	\end{equation}
 	 	We claim that \eqref{Mo9} is equivalent to $K=0$. Indeed, by the symmetry of $K$,
 	\[
 	g(K_u X, Y) = g(K_u Y, X) = 0, \qquad \forall\, X \in T_xM,
 	\]
 	so $K_u X$ is always proportional to $u$. Hence there exists a scalar function $\alpha(X)$ such that
 	\[
 	K_u X = \alpha(X)\, u.
 	\]
 	Since $u$ and $X$ are arbitrary, this yields $K_X u = \beta(u)\, X$ for some scalar $\beta(u)$.  
 	Symmetry then implies $K_u X = K_X u$. Taking $X \perp u$ gives $\alpha(X) = \beta(u) = 0$. Hence,
 	\[
 	K = 0.
 	\]
 	 	This completes the proof.
 \end{proof}
 	 \begin{proposition}
	 Consider a statistical manifold \((M,g,\nabla)\) and equip its tangent bundle \(TM\) with a Cheeger--Gromoll type metric \(g_{p,q}\).  
	 Then the geodesic flow determined by the connection \(\nabla\) has zero divergence (i.e., is incompressible) if and only if one of the following holds:
	 	\begin{enumerate}[i)]
	 		\item the parameters of the metric satisfy \(p=q=0\), so that \(g_{p,q}\) coincides with the Sasaki metric;
	 		\item the difference tensor of the statistical structure vanishes, \(K=0\), equivalently \(\nabla=\widehat{\nabla}\) is the Levi--Civita connection of \((M,g)\).
	 	\end{enumerate}
	 \end{proposition}
 \begin{proof}
 	Recall that the geodesic flow vector field \(\xi\) on \((M,\nabla)\) is locally given by
 	\[
 	\xi=\sum_{i=1}^n u^i\left( \frac{\partial}{\partial x^i} \right)^H .
 	\]
 	A direct computation yields
 	\begin{align*}
 	(\operatorname{div}\xi)_{(x,u)} 
 	&= \operatorname{trace}\!\left\{ X \mapsto -2 K_u X 
 	+ \frac{p}{\alpha} g(K_u u, u)\, X 
 	- q\, g(X,u)\!\left[ K_u u - \frac{q}{1+q\tau} g(K_u u, u)\, u \right] \right\} \\
 	&= -2 \operatorname{trace}(K_u) 
 	+ \frac{n p - \alpha q}{\alpha}\, g(K_u u, u).
 	\end{align*}
 	Thus, \(\xi\) is incompressible if and only if
 	\begin{equation}\label{eq:incompressible}
 	\operatorname{trace}(K_u)
 	= \frac{n p - \alpha q}{2\alpha}\, g(K_u u, u),
 	\qquad \forall u\in TM.
 	\end{equation}
 	 	Evaluating \eqref{eq:incompressible} at \(\lambda u\) gives
 	\[
 	\operatorname{trace}(K_{\lambda u})
 	= \frac{n p - (1+\lambda^2\tau)\, q}{2 (1+\lambda^2\tau)}\, 
 	\lambda^3\, g(K_u u, u),
 	\]
 	or equivalently,
 	\begin{equation}\label{eq:trace_lambda}
 	\operatorname{trace}(K_u)
 	= \lambda^2\, 
 	\frac{n p - (1+\lambda^2\tau)\, q}{2 (1+\lambda^2\tau)}
 	g(K_u u, u),
 	\qquad u\in TM,\; \lambda\in\mathbb{R}.
 	\end{equation}
 	 	Comparing \eqref{eq:incompressible} and \eqref{eq:trace_lambda}, we obtain
 	\begin{equation}\label{Mo10}
 	f(\lambda,\lvert u\rvert^2)\, g(K_u u, u)=0,
 	\qquad \forall u\in TM,\; \lambda\in\mathbb{R},
 	\end{equation}
 	where \(f:\mathbb{R}^+\times\mathbb{R}^+\to\mathbb{R}\) is defined by
 	\begin{equation}\label{Mo11}
 	f(\lambda,\mu)
 	=\frac{n p - (1+\mu)q}{2(1+\mu)}
 	-\frac{\lambda\bigl(n p - (1+\lambda\mu)q\bigr)}{2(1+\lambda\mu)}.
 	\end{equation}
 	 	Assume that there exists \(u\in TM\) such that \(g(K_u u, u)\neq 0\) (hence \(u\neq 0\)).  
 	Then for any \(v\neq 0\), we also have \(g(K_{v u}(v u), v u)\neq 0\).  
 	From \eqref{Mo10} it follows that
 	\[
 	f(\lambda, v^2 \lvert u\rvert^2)=0,
 	\qquad \forall v\in\mathbb{R}^+.
 	\]
 	Therefore \(f(\lambda,\mu)=0\) for all \((\lambda,\mu)\in\mathbb{R}^+\times\mathbb{R}^+\).  
 	Using \eqref{Mo11} this yields
 	\[
 	q\,\lambda\mu^2 + q(1+\lambda)\mu + q - n p = 0,
 	\qquad \forall (\lambda,\mu)\in\mathbb{R}^+\times\mathbb{R}^+,
 	\]
 	which forces \(q=p=0\).
 	 	If instead \(g(K_u u, u)=0\) for all \(u\in TM\), then by symmetry we obtain \(K=0\). 
 \end{proof}
		 	 	\begin{proposition}\label{MA11}
	 	Let \((M,g,\nabla)\) be a statistical manifold. The components of the curvature tensor \(R^{p,q}\) of the Levi-Civita connection \(\nabla^{p,q}\) corresponding to the Cheeger-Gromoll type metric \(g_{p,q}\) at \((x,u) \in BM_q\) are given by
	 			 			\begin{align*}
	 			\big\{R^{p,q}(X^H,Y^H)Z^H\big\}^H&=R(X,Y)Z+(\nabla_YK)(X,Z)-(\nabla_XK)(Y,Z)+K_XK_YZ-K_YK_XZ\nonumber\\
	 			&+\frac{1}{4\alpha^p}\{\widetilde R(u,R(X,Z)u)Y+qg(R(X,Z)u,u)\widetilde R(u,u)Y-\widetilde R(u,R(Y,Z)u)X\nonumber\\
	 			&-qg(R(Y,Z)u,u)\widetilde R(u,u)X+2\widetilde R(u,R(X,Y)u)Z+2qg(R(X,Y)u,u)\widetilde R(u,u)Z\},\nonumber
	 			\end{align*}
	 			\begin{align*}
	 	\big\{R^{p,q}(X^H,Y^H)Z^H\big\}^V&=\frac{1}{2}((\nabla_YR)(X,Z,u)-(\nabla_XR)(Y,Z,u))	+\frac{1}{2}\{K_XR(Y,Z)u-R(Y,K_XZ)u\nonumber\\
 		&-K_YR(X,Z)u+R(X,K_YZ)u\}+\frac{p}{2\alpha}\{g(K_Yu,u)R(X,Z)u-g(K_Xu,u)R(Y,Z)u\nonumber\\
 		&+2g(K_Zu,u)R(X,Y)u\}+\frac{q}{2}g(R(Y,Z)u,u)\{K_Xu-\frac{q}{1+q\tau}g(K_Xu,u)u\}\nonumber\\
 		&-\frac{q}{2}g(R(X,Z)u,u)\{K_Yu-\frac{q}{1+q\tau}g(K_Yu,u)u\}-qg(R(X,Y)u,u)\{K_Zu\nonumber\\
 		&-\frac{q}{1+q\tau}g(K_Zu,u)u\}-K_{R(X,Y)u}Z\nonumber,
	 			\end{align*}
\begin{align*}
	\big\{R^{p,q}(X^H,Y^H)Z^V\big\}^H&=
-\frac{p}{\alpha}(K_uu)g_{p,q}(Z^V,(R(X,Y)u)^V)+\frac{1}{\alpha^p}\{(K_{R(X,Y)u}Z)+q(K_Zu)g(R(X,Y)u,u)	\\
	&+q(K_{R(X,Y)u}u)g(Z,u)\}+\frac{1}{2\alpha^p}\{(\nabla_X\widetilde R)(u,Z,Y)-(\nabla_Y\widetilde R)(u,Z,X)\}\\
	&+\frac{q}{\alpha^p}\{g(K_YZ,u)\widetilde R(u,u)X-g(K_XZ,u)\widetilde R(u,u)Y\}+\frac{q}{2\alpha^p}g(Z,u)\{(\nabla_X\widetilde R)(u,u,Y)\\
	&-(\nabla_Y\widetilde R)(u,u,X)\}+\frac{pq}{2\alpha^{p+1}}g(Z,u)\{g(K_Xu,u)\widetilde R(u,u)Y-g(K_Yu,u)\widetilde R(u,u)X\}\\
	&+\frac{p}{2\alpha^{p+1}}\{g(K_Xu,u)\widetilde R(u,Z)Y-g(K_Yu,u)\widetilde R(u,Z)X\}+\frac{1}{2\alpha^{p}}\{(K_Y\widetilde R(u,Z)X)\\
	&-(K_X\widetilde R(u,Z)Y)\}+\frac{q}{2\alpha^{p}}g(Z,u)\{(K_Y\widetilde R(u,u)X)-(K_X\widetilde R(u,u)Y)\}\\
	&+\frac{1}{2\alpha^{p}}\{\widetilde R(u,K_XZ)Y+qg(K_XZ,u)\widetilde R(u,u)Y-\widetilde R(u,K_YZ)X-qg(K_YZ,u)\widetilde R(u,u)X\}\\
	&+\frac{q}{2\alpha^p}g(Z,u)\{\widetilde R(u,K_Xu)Y-\widetilde R(u,K_Yu)X\},
	\end{align*}
	\begin{align*}
	\big\{R^{p,q}(X^H,Y^H)Z^V\big\}^V&=R(X,Y)Z-\frac{p}{\alpha}\{g(R(X,Y)u,u)Z+g(Z,u)(R(X,Y)u)\}\\
	&+\mathcal{M}g(R(X,Y)u,Z)u+\mathcal{N}g(R(X,Y)u,u)g(Z,u)u\\
	&+\frac{1}{4\alpha^p}\{R(Y,\widetilde R(u,Z)X)u+qg(Z,u)R(Y,\widetilde R(u,u)X)u-R(X,\widetilde R(u,Z)Y)u\\
	&-qg(Z,u)R(X,\widetilde R(u,u)Y)u\}+qg(Z,u)\{(\nabla_YK)(X,u)-(\nabla_XK)(Y,u)\}\\
	&+(\nabla_YK)(X,Z)-(\nabla_XK)(Y,Z)+K_XK_YZ-K_YK_XZ+qg(K_Xu,Z)\{K_Yu\\
	&-\frac{q}{1+q\tau}g(K_Yu,u)u\}-qg(K_Yu,Z)\{K_Xu-\frac{q}{1+q\tau}g(K_Xu,u)u\}\\
	&+(\frac{q^2}{1+q\tau}g(Z,u)u+\frac{p}{\alpha}Z)\{2g(K_YK_Xu,u)-g((\nabla_YK)(X,u),u)-2g(K_XK_Yu,u)\\
	&+g((\nabla_XK)(Y,u),u)\}+qg(Z,u)(K_XK_Yu-K_YK_Xu),
\end{align*}
\begin{align*}
\big\{R^{p,q}(X^H,Y^V)Z^H\big\}^H&=\frac{1}{2\alpha^p}(\nabla_X\widetilde R)(u,Y,Z)+\frac{q}{2\alpha^p}g(Y,u)(\nabla_X\widetilde R)(u,u,Z)-\frac{q}{\alpha^p}g(K_XY,u)\widetilde R(u,u)Z\\
&+\frac{p}{\alpha^{p+1}}g(K_Xu,u)\{\widetilde R(u,Y)Z+qg(Y,u)\widetilde R(u,u)Z\}-\frac{1}{2\alpha^p}(K_X\widetilde R(u,Y)Z)\\
&+\frac{pq}{2\alpha^{p+1}}g(K_Zu,u)g(Y,u)\widetilde R(u,u)X-\frac{q}{2\alpha^p}g(Y,u)(K_X\widetilde R(u,u)Z)\\
&+\frac{p}{2\alpha^{p+1}}g(K_Zu,u)\widetilde R(u,Y)X-\frac{1}{2\alpha^p}\{\widetilde R(u,K_YZ)X+qg(K_YZ,u)\widetilde R(u,u)X \}\\
&-\frac{q}{2\alpha^p}g(Y,u)\widetilde R(u,K_Zu)X+\frac{1}{2\alpha^ p}\{\widetilde R(u,Y)K_XZ+qg(Y,u)\widetilde R(u,u)K_XZ \}\\
&-\frac{p}{2\alpha}(K_uu)g_{p,q}(Y^V,(R(X,Z)u)^V)+\frac{1}{2\alpha^p}(K_Y R(X,Z)u)+\frac{q}{2\alpha^p}g(Y,u)(K_uR(X,Z)u)\\
&+\frac{q}{2\alpha^p}(K_Yu)g(R(X,Z)u,u),
\end{align*}
\begin{align*}
\big\{R^{p,q}(X^H,Y^V)Z^H\big\}^V&=2qg(K_XY,u)\{K_Zu-\frac{q}{1+q\tau}g(K_Zu,u)u\}-(\nabla_XK)(Y,Z)\\
&+(\frac{q^2}{1+q\tau}g(Y,u)u+\frac{p}{\alpha}Y)\{g((\nabla_XK)(Z,u),u)-2g(K_XK_Zu,u)\}\\
&-qg(Y,u)(\nabla_XK)(Z,u)+(\frac{2pq^2}{\alpha(1+q\tau)}+\frac{q^3}{(1+q\tau)^2})g(Y,u)g(K_Zu,u)g(K_Xu,u)u\\
&+\frac{p(2+p)}{\alpha^2}g(K_Xu,u)g(K_Zu,u)Y+qg(K_YZ,u)\{K_Xu-\frac{q}{1+q\tau}g(K_Xu,u)u\}\\
&-\frac{qp}{\alpha}g(Y,u)g(K_Zu,u)K_Xu-\frac{p}{\alpha}g(K_Xu,u)K_YZ-qg(Y,u)\{\frac{p}{\alpha}g(K_Xu,u)K_Zu\\
&-K_XK_Zu\}-\frac{p}{\alpha}g(K_Zu,u)K_XY+\frac{p}{\alpha}g(K_uK_XZ,u)Y+K_XK_YZ-K_YK_XZ\\
&-qg(Y,u)\{K_uK_XZ-\frac{q}{1+q\tau}g(K_uK_XZ,u)u\}-\frac{1}{4\alpha^p}R(X,\widetilde R(u,Y)Z)u\\
&-\frac{q}{4\alpha^p}g(Y,u)R(X,\widetilde R(u,u)Z)u-\frac{p}{2\alpha}\{g(Y,u)R(X,Z)u+g(R(X,Z)u,u)Y\}\\
&+\frac{1}{2\alpha(1+q\tau)}\{(q\alpha+p)g(Y,R(X,Z)u)+pqg(Y,u)g(R(X,Z)u,u)\}u+\frac{1}{2}R(X,Z)Y,
\end{align*}
\begin{align*}
\big\{R^{p,q}(X^H,Y^V)Z^V\big\}^H=&-\frac{p(p+2)}{\alpha^2}g(K_Xu,u)(K_uu)g_{p,q}(Y^V,Z^V)+\frac{p}{2\alpha^{p+1}}g(Y,u)\widetilde R(u,Z)X+\frac{1}{2\alpha^{p+1}}(pg(Y,Z)\\
&+pqg(Y,u)g(Z,u))
\widetilde R(u,u)X-\frac{1}{2\alpha^p}\widetilde R(Y,Z)X-\frac{p+q\alpha}{2\alpha^{p+1}}g(Z,u)\widetilde R(u,Y)X\\
&-\frac{q}{2\alpha^{p}}g(Z,u)\widetilde R(Y,u)X-\frac{1}{(2\alpha^p)^2}\{\widetilde R(u,Y)\widetilde R(u,Z)X+qg(Y,u)\widetilde R(u,u)\widetilde R(u,Z)X\}\\
&
-\frac{q}{\alpha^p}g(K_XZ,u)K_Yu-\frac{q}{(2\alpha^p)^2}g(Z,u)\{\widetilde R(u,Y)\widetilde R(u,u)X+qg(Y,u)\widetilde R(u,u)\widetilde R(u,u)X\}\\
&
+\frac{1}{\alpha^p}(\nabla_XK)(Y,Z)-\frac{p}{\alpha}g_{p,q}(Y^V,Z^V)(\nabla_XK)(u,u)+\frac{q}{\alpha^p}\{g(Y,u)(\nabla_XK)(Z,u)\\
&+g(Z,u)(\nabla_XK)(Y,u)\}-\frac{2q}{\alpha^p}g(K_XY,u)K_Zu+\frac{pq}{\alpha^{p+1}}g(Y,u)g(K_Xu,u)K_Zu\\
&+\frac{1}{\alpha^p}(K_YK_XZ-K_XK_YZ)+\frac{q}{\alpha^p}g(Z,u)(K_YK_Xu-K_XK_Yu)\\
&-\frac{q^2(q\alpha-p(1+q\tau))}{(1+q\tau)\alpha^{p+1}}g(Y,u)g(Z,u)g(K_Xu,u)K_uu-\frac{p}{\alpha}K_uu\{-2g_{p,q}((K_XY)^V,Z^V)\\
&-\frac{2q}{\alpha^p}g(Y,u)g(K_XZ,u)+g_{p,q}(Y^V,(K_XZ)^V)+qg(Z,u)g_{p,q}(Y^V,(K_Xu)^V)\}\\
&+\frac{q}{\alpha^p}g(Y,u)\{K_uK_XZ+qg(Z,u)K_uK_Xu\}+\frac{p}{\alpha^{p+1}}g(K_Xu,u)K_YZ\\
&-\frac{q}{\alpha^p}g(Y,u)K_XK_Zu+\frac{pq}{\alpha^{p+1}}g(K_Xu,u)g(Z,u)K_Yu
+\frac{p}{\alpha}g_{p,q}(Y^V,Z^V)K_XK_uu,
\end{align*}
\begin{align*}
\big\{R^{p,q}(X^H,Y^V)Z^V\big\}^V=&
-\frac{1}{2\alpha^p}R(X,K_YZ)u+\frac{p}{2\alpha}g_{p.q}(Y^V,Z^V)R(X,K_uu)u-\frac{q}{2\alpha^p}\{g(Y,u)R(X,K_Zu)u\\
&+g(Z,u)R(X,K_Yu)u\}-\frac{1}{2\alpha^p}\{\frac{p}{\alpha}g(K_u\widetilde R(u,Z)X,u)Y-(K_Y\widetilde R(u,Z)X)\\
&-qg(Y,u)\{(K_u\widetilde R(u,Z)X)-\frac{q}{1+q\tau}g(K_u\widetilde R(u,Z)X,u)u\}\}\\
&-\frac{q}{2\alpha^p}g(Z,u)\{\frac{p}{\alpha}g(K_u\widetilde R(u,u)X,u)Y-(K_Y\widetilde R(u,u)X)-qg(Y,u)\{(K_u\widetilde R(u,u)X)\\
&-\frac{q}{1+q\tau}g(K_u\widetilde R(u,u)X,u)u\}\}+\frac{p}{\alpha}g(K_XZ,u)Y-\mathcal{M}g(K_XY,Z)u\\
&-\mathcal{N}g(K_XZ,u)g(Y,u)u+(\mathcal{M}q-2\mathcal{N}-\frac{2q^2}{1+q\tau})g(K_XY,u)g(Z,u)u\\
&+(\frac{p}{\alpha}+q)g(Z,u)(K_XY)+(-2\mathcal{M}'+(\mathcal{M}\tau-1)\frac{q^2}{1+q\tau})g(K_Xu,u)g(Y,Z)u\\
&-\frac{p}{\alpha}g(Y,Z)(K_Xu)+(\frac{2q^3}{(1+q\tau)^2}-2\mathcal{N}'-\frac{q^2\mathcal{M}}{1+q\tau})g(Y,u)g(Z,u)g(K_Xu,u)u\\
&-\frac{2p(1+q\tau)+pq\alpha+q^2\alpha^2}{\alpha^2(1+q\tau)}g(K_Xu,u)g(Z,u)Y,
\end{align*}
\begin{align*}
\big\{	R^{p,q}(X^V,Y^V)Z^H\big\}^H&=\frac{1}{\alpha^p}\widetilde R(X,Y)Z-\frac{2p+q\alpha}{2\alpha^{p+1}}\{g(X,u)\widetilde R(u,Y)Z-g(Y,u)\widetilde R(u,X)Z\}\\
	&-\frac{q}{2\alpha^{p}}\{g(X,u)\widetilde R(Y,u)Z-g(Y,u)\widetilde R(X,u)Z\}+\frac{1}{(2\alpha^p)^2}\{\widetilde R(u,X)\widetilde R(u,Y)Z\\
	&+qg(X,u)\widetilde R(u,u)\widetilde R(u,Y)Z-\widetilde R(u,Y)\widetilde R(u,X)Z-qg(Y,u)\widetilde R(u,u)\widetilde R(u,X)Z\}\\
	&+\frac{q}{(2\alpha^p)^2}\{g(Y,u)\widetilde R(u,X)\widetilde R(u,u)Z-g(X,u)\widetilde R(u,Y)\widetilde R(u,u)Z\}+\frac{1}{\alpha^p}\{K_YK_XZ\\
	&-K_XK_YZ\}-\frac{q}{\alpha^p}\{g(X,u)K_{u}K_YZ-g(Y,u)K_{u}K_XZ+K_Xug(K_YZ,u)\\
	&-K_Yug(K_XZ,u)+g(Y,u)K_{X}K_Zu-g(X,u)K_{Y}K_Zu\}+\frac{p}{\alpha}K_uu\{g_{p,q}(X^V,(K_YZ)^V)\\
	&+qg(Y,u)g_{p,q}(X^V,(K_Zu)^V)-g_{p,q}(Y^V,(K_XZ)^V)-qg(X,u)g_{p,q}(Y^V,(K_Zu)^V)\},
	\end{align*}
	\begin{align*}
	\big\{	R^{p,q}(X^V,Y^V)Z^H\big\}^V&=\frac{p}{2\alpha^{p+1}}\{g(K_u\widetilde R(u,Y)Z,u)X-g(K_u\widetilde R(u,X)Z,u)Y\}+\frac{1}{2\alpha^p}\{(K_Y\widetilde R(u,X)Z)\\
	&-(K_X\widetilde R(u,Y)Z)\}+\frac{q}{2\alpha^p}g(X,u)\{\frac{q}{1+q\tau}g(K_u\widetilde R(u,Y)Z,u)u-(K_u\widetilde R(u,Y)Z)\}\\
	&+\frac{q}{2\alpha^p}g(Y,u)\{(K_u\widetilde R(u,X)Z)-\frac{q}{1+q\tau}g(K_u\widetilde R(u,X)Z,u)u\}\\
	&+\frac{pq}{2\alpha^{p+1}}g(K_u\widetilde R(u,u)Z,u)\{g(Y,u)X-g(X,u)Y\}+\frac{q}{2\alpha^{p}}\{g(X,u)(K_Y\widetilde R(u,u)Z)\\
	&-g(Y,u)(K_X\widetilde R(u,u)Z)\}+(\frac{p+q\alpha}{\alpha})\{g(X,u)K_YZ-g(Y,u)K_XZ\}\\
	&-\frac{p}{\alpha}\{g(K_YZ,u)X-g(K_XZ,u)Y\}+\mathcal{M}\{g(Y,K_XZ)-g(X,K_YZ)\}u\\
	&+(\frac{pq+q^2\alpha}{\alpha(1+q\tau)}+\frac{2p}{\alpha^2})g(K_Zu,u)\{g(Y,u)X-g(X,u)Y\}\\
	&+\frac{q^2}{1+q\tau}\{g(Y,u)g(K_ZX,u)-g(X,u)g(K_ZY,u)\}u,
	\end{align*}
	\begin{align*}
	\big\{R^{p,q}(X^V,Y^V)Z^V\big\}^H&=\frac{1}{2\alpha^{2p}}
	\{\widetilde R(u,X)K_YZ+qg(X,u)\widetilde R(u,u)K_YZ-\widetilde R(u,Y)K_XZ-qg(Y,u)\widetilde R(u,u)K_XZ\}\\
	&+\frac{p}{2\alpha^{p+1}}\{\{(\widetilde R(u,Y)K_uu)+qg(Y,u)(\widetilde R(u,u)K_uu)\}g_{p,q}(X^V,Z^V)-\{(\widetilde R(u,X)K_uu)\\
	&+qg(X,u)(\widetilde R(u,u)K_uu)\}g_{p,q}(Y^V,Z^V)\}+\frac{q}{2\alpha^{2p}}g(Z,u)\{\widetilde R(u,X)K_Yu-\widetilde R(u,Y)K_Xu\}\\
	&+\frac{q}{2\alpha^{2p}}\{g(Y,u)\widetilde R(u,X)K_Zu-g(X,u)\widetilde R(u,Y)K_Zu\}+\frac{q^2}{2\alpha^{2p}}g(Z,u)\{g(X,u)\widetilde R(u,u)K_Yu\\
	&-g(Y,u)\widetilde R(u,u)K_Xu\}+(\frac{2p}{\alpha^{p+2}}+\frac{\mathcal{M}q}{\alpha^{p}})K_uu\big(g(X,u)g(Y,Z)-g(Y,u)g(X,Z)\big)\\
	&+\frac{p+q\alpha}{\alpha^{p+1}}\{g(Y,u)K_XZ-g(X,u)K_YZ\}-\frac{p}{\alpha^{p+1}}\{g(Y,Z)K_Xu-g(X,Z)K_Yu\},
	\end{align*}
	\begin{align*}
	\big\{R^{p,q}(X^V,Y^V)Z^V\big\}^V&=	-\frac{p^2}{\alpha^2}g(K_uK_uu,u)\{g_{p,q}(Y^V,Z^V)X-g_{p,q}(X^V,Z^V)Y\}\\
	&+\frac{p}{\alpha}g_{p,q}(Y^V,Z^V)\{K_XK_uu+qg(X,u)\{K_uK_uu\\
	&-\frac{q}{1+q\tau}g(K_uK_uu,u)u\}\}-\frac{p}{\alpha}g_{p,q}(X^V,Z^V)\{K_YK_uu+qg(Y,u)\{K_uK_uu\\
	&-\frac{q}{1+q\tau}g(K_uK_uu,u)u\}\}+\frac{p}{\alpha^{p+1}}\{g(K_uK_YZ,u)X-g(K_uK_XZ,u)Y\}\\
	&+\frac{1}{\alpha^p}(K_YK_XZ-K_XK_YZ)-\frac{q}{\alpha^p}g(X,u)\{(K_uK_YZ)-\frac{q}{1+q\tau}g(K_uK_YZ,u)u\}\\
	&+\frac{q}{\alpha^p}g(Y,u)\{(K_uK_XZ)-\frac{q}{1+q\tau}g(K_uK_XZ,u)u\}+\frac{pq}{\alpha^{p+1}}g(K_uK_Zu,u)\{g(Y,u)X\\
	&-g(X,u)Y\}+\frac{q}{\alpha^p}\{g(X,u)K_YK_Zu-g(Y,u)K_XK_Zu\}\\
	&+\frac{q}{\alpha^p}g(Z,u)\{\frac{p}{\alpha}g(K_uK_Yu,u)X-\frac{p}{\alpha}g(K_uK_Xu,u)Y+(K_YK_Xu-K_XK_Yu)\\
	&-qg(X,u)\{(K_uK_Yu)-\frac{q}{1+q\tau}g(K_uK_Yu,u)u\}+qg(Y,u)\{(K_uK_Xu)\\
	&-\frac{q}{1+q\tau}g(K_uK_Xu,u)u\}\}\!+\!\frac{(\tau \mathcal{N}\alpha+2-p)p-\mathcal{N}\alpha^2}{\alpha^2}g(Z,u)\!(g(X,u)Y\\
	&-g(Y,u)X)+(2\mathcal{M}'+\mathcal{M}(\mathcal{M}+\mathcal{N}\tau)-\mathcal{N})\{g(Y,Z)g(X,u)-g(X,Z)g(Y,u)\}u\\
	&+\frac{p(\mathcal{M}\tau-1)-\mathcal{M}\alpha}{\alpha}(g(X,Z)Y-g(Y,Z)X).
	\end{align*}
\end{proposition}
\begin{proof}
	We focus on proving the case  $\big(	R^{p,q}(X^V,Y^V)Z^V\big)^V$, as other cases follow similarly. 
		The vertical component of the curvature tensor is given by:
	\begin{align}\label{L2222}
	\big(	R^{p,q}(X^V,Y^V)Z^V\big)^V=&\big(	{{\nabla}}^{p,q}_{X^V}	{{\nabla}}^{p,q}_{Y^V}Z^V-	{{\nabla}}^{p,q}_{Y^V}	{{\nabla}}^{p,q}_{X^V}Z^V\big)^V.
	\end{align}
By part (iv) of Theorem \ref{L2}, the first term on the right-hand side of \eqref{L2222} can be written as
	\begin{align*}
	\big(	{{\nabla}}^{p,q}_{X^V}	{{\nabla}}^{p,q}_{Z^V}Z^V\big)^V	=&\big(	{{\nabla}}^{p,q}_{Y^V}[-\frac{p}{\alpha}
	(K_uu)^Hg_{p,q}(Z^V,Y^V)+\frac{1}{\alpha^p}(K_YZ)^H+\frac{q}{\alpha^p}(K_Zu)^Hg(Y,u)+\frac{q}{\alpha^p}(K_Yu)^Hg(Z,u)\\
	&
 -\frac{p}{\alpha}(g(Y,u)Z^V+g(Z,u)Y^V)+\mathcal{M}g(Y,Z)u^V  +\mathcal{N}g(Y,u)g(Z,u)u^V]\big)^V.
	\end{align*}
	For each term on the right-hand side, applying part (iv) of Theorem \ref{L2} gives:
	\begin{align*}
	\big(	{{\nabla}}^{p,q}_{X^V}[-\frac{p}{\alpha}
	(K_uu)^Hg_{p,q}(Z^V,Y^V)]\big)^V=&-\frac{p}{\alpha}
	g_{p,q}(Z^V,Y^V)[\frac{p}{\alpha}g(K(	{K_uu},u),u)X^V-(K(X,{K_uu}))^V\\
	& -qg(X,u)\{(K({K_uu},u))^V-\frac{q}{1+q\tau}g(K({K_uu},u),u)u^V\}],
	\end{align*}
	\begin{align*}
	\big(	{{\nabla}}^{p,q}_{X^V}[\frac{1}{\alpha^p}(K_YZ)^H]\big)^V=&\frac{p}{\alpha^{p+1}}g(K(K_YZ,u),u)X^V-\frac{1}{\alpha^p}(K(X,K_YZ))^V\\
	&  -\frac{q}{\alpha^p}g(X,u)\{(K(K_YZ,u))^V-\frac{q}{1+q\tau}g(K(K_YZ,u),u)u^V\},\hspace{2cm}
		\end{align*}
	\begin{align*}
	\big(	{{\nabla}}^{p,q}_{X^V}[\frac{q}{\alpha^p}(K_Zu)^Hg(Y,u)]\big)^V=&\frac{q}{\alpha^p}g(Y,u)[\frac{p}{\alpha}g(K(K_Zu,u),u)X^V-(K(X,K_Zu))^V\\
	& -qg(X,u)\{(K(K_Zu,u))^V-\frac{q}{1+q\tau}g(K(K_Zu,u),u)u^V\}],\hspace{1.2cm}
	\end{align*}
	\begin{align*}
	\big(	{{\nabla}}^{p,q}_{X^V}[\frac{q}{\alpha^p}(K_Yu)^Hg(Z,u)]\big)^V=&\frac{q}{\alpha^p}g(Z,u)[\frac{p}{\alpha}g(K(K_Yu,u),u)X^V-(K(X,K_Yu))^V\\
	& -qg(X,u)\{(K(K_Yu,u))^V-\frac{q}{1+q\tau}g(K(K_Yu,u),u)u^V\}],\hspace{1cm}
	\end{align*}
	\begin{align*}
	\big(	{{\nabla}}^{p,q}_{X^V}[-\frac{p}{\alpha}(g(Y,u)Z^V+g(Z,u)Y^V)]\big)^V=&\frac{2p}{\alpha^2}g(X,u)(g(Y,u)Z^V+g(Z,u)Y^V)-\frac{p}{\alpha}g(Y,u)[-\frac{p}{\alpha}(g(X,u)Z^V\\
	&+g(Z,u)X^V)+\mathcal{M}g(X,Z)u^V  +\mathcal{N}g(X,u)g(Z,u)u^V]\\
	&-\frac{p}{\alpha}g(Z,u)[-\frac{p}{\alpha}(g(X,u)Y^V+g(Y,u)X^V)+\mathcal{M}g(X,Y)u^V \\
	& +\mathcal{N}g(X,u)g(Y,u)u^V]-\frac{p}{\alpha}(g(Y,X)Z^V+g(Z,X)Y^V),
	\end{align*}
	\begin{align*}		\big(	{{\nabla}}^{p,q}_{X^V}[\mathcal{M}g(Y,Z)u^V]\big)^V=&2\mathcal{M'}g(X,u)g(Y,Z)u^V+\mathcal{M}g(Y,Z)[-\frac{p}{\alpha}(g(X,u)u^V+\tau X^V)\\
	&+\mathcal{M}g(X,u)u^V  +\tau\mathcal{N}g(X,u)u^V]+\mathcal{M}g(Y,Z)X^V,\hspace{4cm}
	\end{align*}
	\begin{align*}
	\big(	{{\nabla}}^{p,q}_{X^V}[\mathcal{N}g(Y,u)g(Z,u)u^V]\big)^V=&2\mathcal{N'}g(X,u)g(Y,u)g(Z,u)u^V+\mathcal{N}g(Y,u)g(Z,u)X^V\\
	&+\mathcal{N}g(Y,u)g(Z,u)[-\frac{p}{\alpha}(g(X,u)u^V+\tau X^V)+\mathcal{M}g(X,u)u^V  \\
	&+\tau\mathcal{N}g(X,u)u^V]+\mathcal{N}g(Y,X)g(Z,u)u^V+\mathcal{N}g(Y,u)g(Z,X)u^V.\hspace{1cm}
	\end{align*}
	Summing these seven terms, performing the analogous calculation with $X$ and $Y$ switched to compute $\big( \nabla^{p,q}_{Y^V} \nabla^{p,q}_{X^V} Z^V \big)^V$, and subtracting the two results, we obtain the explicit expression for the vertical component $\big(R^{p,q}(X^V,Y^V)Z^V\big)^V$ as claimed.
 	\end{proof}
 \section{Sectional and  scalar curvatures of {$(TM, g_{p,q})$}}
 Let $(M, g, \nabla)$ be a statistical manifold, and let $|\cdot|$ denote the norm induced by the Riemannian metric $g$. Given any pair of vector fields $X, Y \in \mathfrak{X}(M)$, the squared area of the parallelogram spanned by $X$ and $Y$ is defined by
 \[
 Q(X, Y) := |X|^2 |Y|^2 - g(X, Y)^2.
 \]
 This quantity vanishes if and only if $X$ and $Y$ are linearly dependent. If $X$ and $Y$ are linearly independent, the {sectional curvature} of the $2$-plane they span is given by
 \[
 \mathcal{K}(X, Y) := \frac{g(R(X, Y)Y, X)}{Q(X, Y)},
 \]
 where $R$ is the curvature tensor associated with the connection $\nabla$.

  \begin{theorem}\label{N2}
	Let $(M,g, \nabla)$ be a statistical manifold and $TM$ be its tangent bundle endowed with the Cheeger-Gromoll type metric $g_{p,q}$. Then $(TM,g_{p,q})$ is of constant sectional curvature if and only if it is flat and the following assertions hold:
	\begin{itemize}
		\item [(i)] $(M,g)$ is flat;
		\item [(ii)] $K$ is parallel with respect to $\nabla$;
		\item [(iii)] $[K_X,K_Y]=0$, for all $X,Y \in \chi(M)$;
		\item [(iv)] $p=q=0$, i.e. $g_{p,q}$ is the Sasaki metric.
	\end{itemize}
\end{theorem}
\begin{proof}
	It is well known that $(TM,g_{p,q})$ is of constant sectional curvature $\kappa$ if and only if
	$$R^{p,q}(X,y)Z=\kappa(g(Y,Z)X -g(X,Z)Y), \qquad \forall X,Y,Z \in \chi(TM).$$
	Suppose that $(TM,g_{p,q})$ is of constant sectional curvature $\kappa$. Then the six expressions of Propostion \ref{MA11}, restricted to the zero section of $TM$, are equivalent to the following identities
	\begin{eqnarray}
	& & R(X,Y)Z +(\nabla_Y K)(X,Z) -(\nabla_X K)(Y,Z) +[K_X,K_Y]Z =\kappa(g(Y,Z)X -g(X,Z)Y), \label{flat1}\\
	& & R(X,Y)Z +(\nabla_Y K)(X,Z) -(\nabla_X K)(Y,Z) +[K_X,K_Y]Z =0,\label{flat2}\\
	& & \frac12 R(X,Z)Y-(\nabla_X K)(Y,Z) +[K_X,K_Y]Z=-\kappa g(X,Z)Y,\label{flat3}\\
	& &- \frac12 \tilde{R}(Y,Z)X +(\nabla_X K)(Y,Z) -[K_X,K_Y]Z= \kappa g(Y,Z)X, \label{flat4}\\
	& & \tilde{R}(X,Y)Z - [K_X,K_Y]Z=0,\label{flat5}\\
	& & [K_X,K_Y]Z -(2p+q)(g(Y,Z)X -g(X,Z)Y) =-\kappa (g(Y,Z)X -g(X,Z)Y),\label{flat6}
	\end{eqnarray}
	for all $X,Y,Z \in \chi(M)$.
		Comparing \eqref{flat1} and \eqref{flat2}, it is easy to deduce that $\kappa =0$, i.e. $(TM,g_{p,q})$ is flat. Now, substituting from \eqref{flat3} into \eqref{flat4}, with $\kappa=0$, we obtain
	$$-\tilde{R}(Y,Z)X +R(X,Z)Y=0.$$
	Then, for all $W \in \chi(M)$, we have
	\begin{equation}\label{flat7}
	0=g(\tilde{R}(Y,Z)X,W) +g(R(X,Z)Y,W)= g(R(X,W)Y,Z) +g(R(X,Z)Y,W).
	\end{equation}
	On the other hand, combining \eqref{flat5} and \eqref{flat6}, we get
	\begin{equation}\label{flat8}
	\tilde{R}(X,Y)Z = (2p+q)(g(Y,Z)X -g(X,Z)Y).
	\end{equation}
	Substituting from \eqref{flat8} into \eqref{flat7}, we obtain
	$$(2p+q)(g(Y,Z)g(X,W) -g(X,Z)g(Y,W))=0,$$
	for all $X,Y,Z,W \in \chi(M)$, thus
	\begin{equation}\label{flat9}
	2p+q=0.
	\end{equation}
	Hence from \eqref{flat8}, it follows that
	\begin{equation}\label{flat10}
	R=0.
	\end{equation}
	Substituting from \eqref{flat10} into \eqref{flat5} then into \eqref{flat3}, we get
	\begin{equation}\label{flat11}
	[K_X,K_Y] =0, \ \ \  \nabla K=0.
	\end{equation}
	Now, taking in the expression of $R^{p,q}(X^H,Y^V)Z^V$ of Proposition \ref{MA11} with $Z=u \neq 0$ and $Y \perp u$, $Y \neq 0$, then the vanishing of the vertical part of right hand side gives
	\begin{equation}\label{flat12}
	\begin{split}
	&\left(\frac{p}{\alpha} -\tau \frac{2p(1+q\tau) +pq\alpha +q^2\alpha^2}{\alpha^2(1+q\tau)}\right)g(K_X,u,u) Y  \\
	= & -\left(\left((\mathcal{M}q-2\mathcal{N} -\frac{2q^2}{1+q\tau}\right)\tau-\mathcal{M}\right) g(K_XY,u) u -\left(\frac{p}{\alpha} +q\right) \tau K_XY.
	\end{split}
	\end{equation}
	The symmetry in $X$ and $Y$ of the right hand side of \eqref{flat12} implies that of the left hand side, i.e.
	\begin{equation}\label{flat13}
	\left(\frac{p}{\alpha} -\tau \frac{2p(1+q\tau) +pq\alpha +q^2\alpha^2}{\alpha^2(1+q\tau)}\right)(g(K_X,u,u) Y -g(K_Y u,u)X)=0.
	\end{equation}
	We have then two cases:
	\\
	\textbf{Case 1:} there is $\tau >0$ such that $\frac{p}{\alpha} -\tau \frac{2p(1+q\tau) +pq\alpha +q^2\alpha^2}{\alpha^2(1+q\tau)} \neq 0$. In this case, $g(K_X,u,u) Y =g(K_Y u,u)X$, for any $X$, $u$ such that $\|u\|^2 =\tau$ and $Y \perp u$. By linearity, $g(K_X,u,u) Y =g(K_Y u,u)X$ holds for any $X$, $u$ and $Y \perp u$. In particular, for $X =u$, we get $g(K_uu,u) =g(K_Y u,u)=0$, for any $u$ and $Y \perp u$. But by the symmetry properties of $K$, we have $g(K_Yu,u)=g(K_uY,u)=g(K_uu,Y)$. We deduce that $g(K_uu,u) =g(K_uu,Y)=0$, for any $u$ and $Y \perp u$. We deduce that $K_uu=0$ for all $u$. By the symmetry of $K$, we have $K=0$. Now, taking in the expression of $R^{p,q}(X^V,Y^V)Z^V$ of Proposition \ref{MA11} $Y=Z=u \neq 0$ and $X \perp u$, $X \neq 0$, then we obtain
	\begin{equation*}
	- \frac{(\tau \mathcal{N} \alpha +2 -p)p -\mathcal{N}\alpha^2}{\alpha^2}\tau^2 -\frac{p(\mathcal{M}\tau -1) -\mathcal{M}\alpha}{\alpha}\tau=0.
	\end{equation*}
	A simple calculation, using $2p+q=0$, yields
	\begin{equation*}
	q\tau^2[(2+q) \tau -q+4]=0, \qquad \textup{for all} \qquad \tau \geq 0,
	\end{equation*}
	which gives $q=0$ and hance $p=0$.
	\\
	\textbf{Case 2:} Suppose that $\frac{p}{\alpha} -\tau \frac{2p(1+q\tau) +pq\alpha +q^2\alpha^2}{\alpha^2(1+q\tau)} =0$, for all $\tau >0$. Since $\alpha=1+\tau$ and $2p+q=0$, then we deduce that $2q^2 \tau^3 +2q^2 \tau^2 +q(2q-1)\tau +q=0$, for all $\tau >0$. It follows that $q=0$, and hence $p=0$.
		Conversely if conditions $(i)-(iv)$ hold then it is easy to check that $R^{p,q}$ vanishes identically.
\end{proof}
\begin{example}\label{Ex0}
	Consider the Euclidean space \(\mathbb{R}^n\) endowed with its standard Riemannian metric 
	\[
	g = dx^{1} \otimes dx^{1} + \cdots + dx^{n} \otimes dx^{n},
	\]
	where \((x^{1}, \ldots, x^{n})\) is the canonical coordinate system. It is well-known that \((\mathbb{R}^n, g)\) is flat.
		Let \(\nabla\) be a statistical connection on \(\mathbb{R}^n\), i.e., 
	\[
	\nabla = \hat{\nabla} + T,
	\]
	where \(\hat{\nabla}\) is the standard flat connection on \(\mathbb{R}^n\), and 
	\[
	T = \sum_{i,j,k=1}^n T^i_{jk} \frac{\partial}{\partial x^i} \otimes dx^j \otimes dx^k,
	\]
	is a \((1,2)\)-tensor field satisfying the symmetry conditions
	\[
	T^i_{jk} = T^i_{kj} = T^j_{ik}, \quad \text{for all } i,j,k = 1, \ldots, n.
	\]
		We equip the tangent bundle \(T\mathbb{R}^n\) with the Sasaki metric \(g_{0,0}\). According to Theorem \ref{N2}, the tangent bundle \((T\mathbb{R}^n, g_{0,0})\) is flat if and only if the tensor \(T\) is parallel with respect to \(\nabla\) and satisfies the commutativity condition \([T_{\partial_i}, T_{\partial_j}] = 0,  i,j=1,\ldots,n\).
		It can be shown that the commutativity condition is equivalent to
	\[
	T^l_{j\lambda} T^\lambda_{ki} = T^l_{k\lambda} T^\lambda_{ji}, \quad \text{for all } i,j,k,l = 1, \ldots, n,
	\]
	and the parallelism of \(T\) translates to the system of partial differential equations
	\[
	\frac{\partial T^i_{jk}}{\partial x^l} = T^i_{jk} T^\lambda_{\lambda l}, \quad \text{for all } i,j,k,l = 1, \ldots, n.
	\]
		To simplify the analysis, we focus on the case \(n=2\). By the symmetry conditions, the tensor \(T\) is completely determined by four functions:
	\[
	A := T^1_{11}, \quad B := T^2_{22}, \quad C := T^2_{11} = T^1_{12} = T^1_{21}, \quad D := T^1_{22} = T^2_{12} = T^2_{21}.
	\]
		Then the commutativity condition reduces to the algebraic relation
	\[
	D (A - D) = C (C - B),
	\]
	while the parallelism condition becomes the system
	\[
	\begin{cases}
	\frac{\partial A}{\partial x} = A^2 + C^2, & \frac{\partial A}{\partial y} = C (A + D), \\
	\frac{\partial B}{\partial x} = D (B + C), & \frac{\partial B}{\partial y} = B^2 + D^2, \\
	\frac{\partial C}{\partial x} = C (A + D), & \frac{\partial C}{\partial y} = C^2 + D^2, \\
	\frac{\partial D}{\partial x} = C^2 + D^2, & \frac{\partial D}{\partial y} = D (B + C).
	\end{cases}
	\]
		For instance, by taking \(C = D = 0\), the algebraic relation is trivially satisfied, and the system reduces to
	\[
	\frac{\partial A}{\partial x} = A^2, \quad \frac{\partial A}{\partial y} = 0, \quad \frac{\partial B}{\partial x} = 0, \quad \frac{\partial B}{\partial y} = B^2,
	\]
	whose general solutions are
	\[
	A(x,y) = \frac{1}{c_1 -x}, \quad B(x,y) = \frac{1}{c_2 - y},
	\]
	for some constants \(c_1, c_2 \in \mathbb{R}\), defined on the open domain 
	\[
	M := \mathbb{R}^2 \setminus \left( \{x = c_1, y = c_2\} \right).
	\]
		Therefore, the statistical manifold \((M, g, \nabla)\) is flat, and its tangent bundle \((TM, g_{0,0})\) with the Sasaki metric is flat as well.
\end{example}
\begin{example}\label{Ex1}
	Let \( M = \mathbb{R}^2 \) be equipped with the pseudo-Riemannian metric 
	\[
	g = dx^1 \otimes dx^2 + dx^2 \otimes dx^1,
	\]
	where \( (x^1, x^2) \) are standard coordinates system. Consider the \((1,2)\)-tensor field
	\[
	K = \widetilde{f}(x^1,x^2) \frac{\partial}{\partial x^1} \otimes dx^2 \otimes dx^2,
	\]
	where \( \widetilde{f} \) is a smooth function on \( \mathbb{R}^2 \). Define a connection \( \nabla \) on \( \mathbb{R}^2 \) by
	\[
	\nabla_{\partial_i} \partial_j := \widehat{\nabla}_{\partial_i} \partial_j + K_{\partial_i} \partial_j,
	\]
	where \( \widehat{\nabla} \) is the Levi-Civita connection of \( g \). It follows that the only non-zero Christoffel symbol is \( \Gamma^1_{22} = \widetilde{f}(x^1,x^2) \) and all other symbols vanish.
		Thus, the components of the cubic form \( \mathcal{C} \) are zero, except 
	\[
	 \mathcal{C}_{222} = -2\widetilde{f}(x^1,x^2),
	\]
	which implies that \( (\mathbb{R}^2, g, \nabla) \) is a statistical manifold.
		From (\ref{2N}), we deduce that the curvature tensor \( R = 0 \), i.e., \( \mathbb{R}^2 \) is flat. Moreover, it can be checked that $	(\nabla_{\partial_i} K)(\partial_j, \partial_k) =0, i,j,k=1,2$, unless
	\[
	(\nabla_{\partial_1} K)(\partial_2, \partial_2) =\partial_1 \widetilde{f}(x^1,x^2)\partial_1,\ \ \ (\nabla_{\partial_2} K)(\partial_2, \partial_2) =\partial_2 \widetilde{f}(x^1,x^2)\partial_1.
	\]
	In addition, one can see that
	\[
K_{\partial_i} K_{\partial_j} \partial_k = K_{\partial_j} K_{\partial_i} \partial_k \quad \forall i,j,k=1,2.
	\]
	So, \( K \) is parallel with respect to \( \nabla \) if and only if \( \widetilde{f} \) is constant. In this case, by Theorem~\ref{N2}, the tangent bundle \( (TM, g_{0,0}) \)  is flat.
\end{example}
\begin{example}
		Let $M$ be the manifold of normal distributions defined by
	\[
	M = \left\{ p(x;\mu,\sigma) = \frac{1}{\sqrt{2\pi}\sigma} \exp\left(-\frac{(x - \mu)^2}{2\sigma^2}\right) \ \middle| \ \mu \in \mathbb{R},\ \sigma > 0 \right\}.
	\]
	This defines a $2$-dimensional statistical manifold with natural coordinates $(\mu, \sigma)$. The Fisher information metric on $M$ is given by
	\begin{align*}
	g = \begin{pmatrix}
	\frac{1}{\sigma^2} & 0 \\
	0 & \frac{2}{\sigma^2}
	\end{pmatrix}.
	\end{align*}
	The non-zero components of the Levi-Civita connection $\widehat{\nabla}$ are
	\begin{align*}
	\widehat{\Gamma}^{1}_{12} = \widehat{\Gamma}^{1}_{21} = -\frac{1}{\sigma}, \quad
	\widehat{\Gamma}^{2}_{11} = \frac{1}{2\sigma}, \quad
	\widehat{\Gamma}^{2}_{22} = -\frac{1}{\sigma},
	\end{align*}
	while the corresponding components of an affine connection $\nabla$ are 
	\begin{align*}
	\Gamma^{1}_{12} = \Gamma^{1}_{21} = -\frac{2}{\sigma}, \quad
	\Gamma^{2}_{22} = -\frac{3}{\sigma},
	\end{align*}
	with all other independent components vanishing.
	Using \eqref{L33}, the non-zero components of the skewness tensor $K$ are as follows
	\begin{align*}
	K^{1}_{12} = K^{1}_{21} = -\frac{1}{\sigma}, \quad
	K^{2}_{11} = -\frac{1}{2\sigma}, \quad
	K^{2}_{22} = -\frac{2}{\sigma}.
	\end{align*}
	Substituting the above equations into \eqref{M4}, the non-zero components of the cubic tensor $\mathcal{C}$ become
	\[
	\mathcal{C}_{111} = 0, \quad
	\mathcal{C}_{112} = \mathcal{C}_{121} = \mathcal{C}_{211} = \frac{2}{\sigma^3}, \quad
	\mathcal{C}_{122} = \mathcal{C}_{212} = \mathcal{C}_{221} = 0, \quad
	\mathcal{C}_{222} = \frac{8}{\sigma^3}.
	\]
	Thus, the triple $(M, g, \nabla)$ is a flat statistical manifold.
The components of the covariant derivative of the tensor \( K \) satisfy
	\[
	\nabla_{\partial_1}K^1_{11} = -\nabla_{\partial_1}K^2_{12} = -\nabla_{\partial_1}K^2_{21} = -\frac{1}{2} \nabla_{\partial_2}K^1_{12} = -\frac{1}{2} \nabla_{\partial_2}K^1_{21} = -\frac{1}{4} \nabla_{\partial_2}K^2_{22} = \frac{1}{\sigma^2}.
	\]
	Hence, the condition~(ii) of Theorem~\ref{N2} doesn't hold, and we conclude that for all values of \( p \) and \( q \), the tangent bundle \( (TM, g_{p,q}) \) is not flat.
	
	
\end{example}
 In the context of the tangent bundle $TM$ equipped with the generalized Cheeger-Gromoll type metric $g_{p,q}$, we denote the corresponding area function and sectional curvature by $Q_{p,q}$ and $\mathcal{K}_{p,q}$, respectively.

Assume $X, Y \in T_x M$ are two orthonormal vectors with respect to $g$, i.e., $g(X, X) = g(Y, Y) = 1$ and $g(X, Y) = 0$. Then, using the definition of $g_{p,q}$ in \eqref{L20}, we obtain:
\begin{align*}
Q_{p,q}(X^H, Y^H) &= 1, \\[5pt]
Q_{p,q}(X^H, Y^V) &= \frac{1}{\alpha^p} \left(1 + q  g(Y, u)^2 \right), \\[5pt]
Q_{p,q}(X^V, Y^V) &= \frac{1}{\alpha^{2p}} \left(1 + q  g(X, u)^2 + q  g(Y, u)^2 \right).
\end{align*}
 	Now, let $(M, g, \nabla)$ be a statistical manifold with sectional curvature $\mathcal{K}$. Using the generalized Cheeger-Gromoll metric $g_{p,q}$ on the tangent bundle $TM$ and applying the geometric framework described in Proposition~\ref{MA11}, we derive explicit expressions for the sectional curvature $\mathcal{K}_{p,q}$ of $(TM, g_{p,q})$.
 	 		\begin{theorem}
 	 			Let $(M, g, \nabla)$ be a statistical manifold with skewness tensor $K$, and let $g_{p,q}$ be the generalized Cheeger-Gromoll type metric on the tangent bundle $TM$. Then the sectional curvature $\mathcal{K}_{p,q}$ of $(TM, g_{p,q})$ satisfies the following expressions:
 	\begin{align}\label{B01}
 	\mathcal{K}_{p,q}(X^H,Y^H)&=\mathcal{K}(X,Y)+g((\nabla_YK)(X,Y),X)-g((\nabla_XK)(Y,Y),X)+g(K_XX,K_YY)\\
 	&-|K_XY|^2-\frac{3}{4\alpha^p}|R(X,Y)u|^2-\frac{3q}{4\alpha^p}(g(R(X,Y)u,u))^2,\nonumber
 	\end{align}
	\begin{align}\label{B1}
\mathcal{K}_{p,q}(X^H,Y^V)=&\frac{\alpha^p}{1+qg(Y,u)^2}\big\{-\frac{1}{(2\alpha^p)^2}\{g(\widetilde R(u,Y)\widetilde R(u,Y)X,X)
\\
&+qg(Y,u)g(\widetilde R(u,u)\widetilde R(u,Y)X,X)\}-\frac{q}{(2\alpha^p)^2}g(Y,u)\{g(\widetilde R(u,Y)\widetilde R(u,u)X,X)
\nonumber\\
&+qg(Y,u)g(\widetilde R(u,u)\widetilde R(u,u)X,X)\}+g(K_Xu,u)^2\{-\frac{p(p+2)}{\alpha^2}|Y^V|^2\nonumber\\&
-\frac{q^3}{(1+q\tau)\alpha^{p}}g(Y,u)^2\}+\frac{2p}{\alpha^{p+1}}g(K_Xu,u)g(K_YY,X)+\frac{1}{\alpha^p}g((\nabla_XK)(Y,Y),X)\nonumber\\
&	-\frac{p}{\alpha}|Y^V|^2g((\nabla_XK)(u,u),X)+\frac{2q}{\alpha^p}g(Y,u)g((\nabla_XK)(Y,u),X)\nonumber\\
&-\frac{3q}{\alpha^p}g(K_XY,u)^2+\frac{4pq}{\alpha^{p+1}}g(Y,u)g(K_Xu,u)g(K_XY,u)\nonumber\\
&+\frac{1}{\alpha^p}g(K_YK_XY-K_XK_YY,X)+\frac{2q}{\alpha^p}g(Y,u)g(K_YK_Xu-K_XK_Yu,X)\nonumber\\
&+\frac{q^2}{\alpha^p}g(Y,u)^2|K_Xu|^2+\frac{p}{\alpha}|Y^V|^2g(K_uu,K_XX)\big\},\nonumber
\end{align}
	\begin{align}\label{Mo1}
	\mathcal{K}_{p,q}(X^V,Y^V)=&\frac{\alpha^{2p}}{1+qg(X,u)^2+qg(Y,u)^2}\big\{-\frac{p^2}{\alpha^2}|K_uu|^2|X^V|^2|Y^V|^2\\
	&+\frac{p}{\alpha^{p+1}}\{|X^V|^2g(K_YY,K_uu)+|Y^V|^2g(K_XX,K_uu)\}\nonumber\\
	&+\frac{2pq}{\alpha^{2p+1}}[g(X,u)g(K_uu,K_uX)+g(Y,u)g(K_uu,K_uY)]\nonumber\\
	&+\frac{1}{\alpha^{2p}}[|K_XY|^2-g(K_YY,K_XX)]+\frac{2q}{\alpha^{2p}}[g(X,u)g(K_XY,K_uY)\nonumber\\
	&+g(Y,u)g(K_XY,K_uX)]-\frac{2q}{\alpha^{2p}}[g(X,u)g(K_YY,K_uX)\nonumber\\
	&+g(Y,u)g(K_XX,K_uY)]+\frac{q^2}{\alpha^{2p}}[g(X,u)^2|K_Yu|^2+g(Y,u)^2|K_Xu|^2]\nonumber\\
	&+\{\frac{p^2q^2}{\alpha^{2+2p}}g(X,u)g(Y,u)|K_uu|^2-\frac{2pq}{\alpha^{1+2p}}g(K_uu,K_YX)\nonumber\\
	&-\frac{2q^2}{\alpha^{2p}}g(K_Xu,K_Yu)\}g(X,u)g(Y,u)\nonumber\\
	&-\!\frac{(\tau \mathcal{N}\alpha+2-p)p-\mathcal{N}\alpha^2}{\alpha^{p+2}}[g(X,u)^2+g(Y,u)^2]-\frac{p(\mathcal{M}\tau-1)-\mathcal{M}\alpha}{\alpha^{p+1}}\}.\nonumber
	\end{align}
	\end{theorem}
\begin{remark}
	The final two terms in the expression for $\mathcal{K}_{p,q}(X^V, Y^V)$ can be rearranged using the following identity:
	\begin{align*}
	\frac{q\tau\left(p(\mathcal{M}\tau - 1) - \mathcal{M}\alpha \right)}{\alpha\tau(1 + q\tau)}
	= 2\mathcal{M}' + \mathcal{M}(\mathcal{M} + \mathcal{N} \tau) - \mathcal{N}
	+ \frac{(\tau \mathcal{N} \alpha + 2 - p)p - \mathcal{N} \alpha^2}{\alpha^2 (1 + q\tau)}.
	\end{align*}
	This decomposition simplifies the expression and separates polynomial and rational components in the curvature formula.
\end{remark}


For a given point $(x,u)\in TM$ with $u\neq0$, let $\{e_1,\ldots,e_m\}$ be an orthonormal basis of the tangent space $T_xM$ with $e_1=\frac{u}{|u|}$. Then we have
$|e_1^V|^2=\frac{1+q\tau}{\alpha^p}$ and  $|e_2^V|^2=\ldots=|e_m^V|^2=\frac{1}{\alpha^p}$. Also,
\begin{align}\label{OM}
\{E_i=e_i^H, E_{m+j}={e_j^V}/{|e_j^V|}|1\leq i\leq m,1\leq j\leq m \},
\end{align}
is an orthonormal basis of the tangent space $T_{(x,u)}TM$ with respect to the Cheeger-Gromoll type metric $g_{p,q}$.
 \begin{lemma}\label{ML}
	Let $(M,g, \nabla)$ be a statistical manifold with the sectional curvature $\mathcal{K}$ such that the tangent space $TM$ is endowed with the Cheeger-Gromoll type metric $g_{p,q}$.
	If $(x,u)\in BM_q$  and $\{E_1,\ldots,E_{2m}\}$ is an orthonormal basis of the tangent space $T_{(x,u)}TM$ characterized by
 (\ref{OM}), then
 the sectional curvature $\mathcal{K}_{p,q}$  is given by
\begin{align}
&\mathcal{K}_{p,q}(E_i,E_j)=\mathcal{K}(e_i,e_j)-\frac{3}{4\alpha^p}|R(e_i,e_j)u|^2-\frac{3q}{4\alpha^p}(g(R(e_i,e_j)u,u))^2+g((\nabla_{e_j}K)(e_i,e_j),e_i)\label{x1}\\
&\ \ \ \ \ \ \ \ \ \ \ \ \ \ \ \ \  -g((\nabla_{e_i}K)(e_j,e_j),e_i)+g(K_{e_i}e_i,K_{e_j}e_j)-|K_{e_i}e_j|^2,\nonumber
\\
&\mathcal{K}_{p,q}(E_i,E_{m+1})=A_1g(K_{e_i}e_1,e_1)^2+A_2g((\nabla_{e_i}K)(e_1,e_1),e_i)+A_3|K_{e_i}e_1|^2+A_4g(K_{e_1}e_1,K_{e_i}e_i)\label{x2}\\
&\ \ \ \ \ \ \ \ \ \ \ \ \ \ \ \ \ \ \ \ \ +A_5|\widetilde R(u,e_1)e_i|^2,\nonumber
\\
& \mathcal{K}_{p,q}(E_i,E_{m+j})=-\frac{1}{4}g(\widetilde R(u,{e_j})\widetilde R(u,{e_j}){e_i},{e_i})
-{p(p+2)}\alpha^{p-2}g(K_{e_i}u,u)^2\label{x3}\\
&\ \ \ \ \ \ \ \ \ \ \ \ \ \ \ \ \ \ \ \ \ +{2p}\alpha^{p-1}g(K_{e_i}u,u)g(K_{e_j}{e_j},{e_i})
	+\alpha^pg((\nabla_{e_i}K)({e_j},{e_j}),{e_i})\nonumber\\
	&\ \ \ \ \ \ \ \ \ \ \ \ \ \ \ \ \ \ \ \ \ -{p}\alpha^{p-1}g((\nabla_{e_i}K)(u,u),{e_i})
-{3q}\alpha^pg(K_{e_i}{e_j},u)^2\nonumber\\
&\ \ \ \ \ \ \ \ \ \ \ \ \ \ \ \ \ \ \ \ \ +\alpha^pg(K_{e_j}K_{e_i}{e_j}-K_{e_i}K_{e_j}{e_j},{e_i})
+{p}\alpha^{p-1}g(K_uu,K_{e_i}{e_i}),\nonumber
\\
&\mathcal{K}_{p,q}(E_{m+1},E_{m+j})=B_1|K_{e_1}e_1|^2+B_2	|K_{e_j}e_1|^2+B_3
g(K_{e_j}e_j,K_{e_1}e_1)+B_4,\label{x4}
\\
&\mathcal{K}_{p,q}(E_{m+i},E_{m+j})=-\alpha^{2p-2}{p^2}|K_uu|^2+\alpha^{2p-1}{p}\{g(K_{e_j}{e_j},K_uu)+g(K_{e_i}{e_i},K_uu)\}\label{x5}\\
&\ \ \ \ \ \ \ \ \ \ \ \ \ \ \ \ \ \ \ \ \ \ \ +\alpha^{2p}|K_{e_i}{e_j}|^2-\alpha^{2p}g(K_{e_j}{e_j},K_{e_i}{e_i})
-\alpha^{p-1}({p(\mathcal{M}\tau-1)-\mathcal{M}\alpha}),\nonumber
\end{align}
where 
\begin{align*}
A_1 &= \frac{\alpha^p(1+q\tau)}{1+q\tau(1+\alpha^p)} \left[
-\frac{p(p+2)\tau^2}{\alpha^2}
-\frac{(q\tau)^3}{(1+q\tau)^2}
+\frac{2p\tau}{\alpha(1+q\tau)}
-\frac{3q\tau}{1+q\tau}
+\frac{4pq\tau^2}{\alpha(1+q\tau)}
\right], \\
A_2 &= \frac{\alpha^p(1+q\tau)}{1+q\tau(1+\alpha^p)} \left(
\frac{1+2q\tau}{1+q\tau} - \frac{p\tau}{\alpha}
\right), \ \ \
A_3 = \frac{\alpha^p(1+q\tau)^2}{1+q\tau(1+\alpha^p)}, \\
A_4 &= -\frac{\alpha^p(1+q\tau)}{1+q\tau(1+\alpha^p)} \left(
\frac{1+2q\tau}{1+q\tau} - \frac{p\tau}{\alpha}
\right), \ \ \
A_5 = \frac{(1+q\tau)^2}{4(1+q\tau(1+\alpha^p))}, 
\end{align*}
and
\begin{align*}
B_1 &= \frac{\alpha^{2p}(1+q\tau)}{1+q\tau(1+\alpha^p)} \left(
-\frac{(p\tau)^2}{\alpha^2}
+\frac{p\tau}{\alpha(1+q\tau)}
+\frac{2pq\tau^2}{\alpha^{p+1}(1+q\tau)}
\right), \\
B_2 &= \frac{\alpha^{2p}(1+q\tau)^2}{1+q\tau(1+\alpha^p)}, \ \ \ 
B_3 = \frac{\alpha^{2p}(1+q\tau)}{1+q\tau(1+\alpha^p)} \left(
\frac{p\tau}{\alpha} - \frac{1+2q\tau}{1+q\tau}
\right), \\
B_4 &= \frac{\alpha^{2p}(1+q\tau)}{1+q\tau(1+\alpha^p)} \left\{
-\frac{(\tau \mathcal{N} \alpha + 2 - p)p - \mathcal{N} \alpha^2}{\alpha^{p+2}} \cdot \frac{\alpha^p \tau}{1+q\tau}
- \frac{p(\mathcal{M}\tau - 1) - \mathcal{M} \alpha}{\alpha^{p+1}}
\right\}.
\end{align*}
\end{lemma}
\begin{proof}
	Putting $X^H=E_i=e_i^H$ and $Y^H=E_j=e_j^H$ in (\ref{B01}) we get (\ref{x1}).
	 Setting $X^H=E_i=e_i^H$ and $Y^V=E_{m+1}=e_1^V/|e_1^V|$ in (\ref{B1}), we have
	\begin{align}\label{c1}
	& \mathcal{K}_{p,q}(E_i,E_{m+1})=\mathcal{K}_{p,q}({e_i}^H,e_1^V/|e_1^V|)=\mathcal{K}_{p,q}(e_i^H,\sqrt{\frac{\alpha^p}{1+q\tau}}e_1^V)\\
	&=\frac{\alpha^p(1+q\tau)}{1+q\tau(1+\alpha^p)}
	\big\{-\frac{1}{4\alpha^p(1+q\tau)}\{g(\widetilde R(u,{e_1})\widetilde R(u,{e_1}){e_i},{e_i})
	+q\sqrt{\tau}g(\widetilde R(u,u)\widetilde R(u,{e_1}){e_i},{e_i})\}\nonumber\\
	&-\frac{\sqrt{\tau}q}{4\alpha^p(1+q\tau)}\{g(\widetilde R(u,{e_1})\widetilde R(u,u){e_i},{e_i})
	+q\sqrt{\tau}g(\widetilde R(u,u)\widetilde R(u,u){e_i},{e_i})\}\nonumber\\
	&+g(K_{e_i}u,u)^2\{-\frac{p(p+2)}{\alpha^2}-\frac{\tau q^3}{(1+q\tau)^2}\}+\frac{2p}{\alpha({1+q\tau})}g(K_{e_i}u,u)g(K_{e_1}{e_1},{e_i})\nonumber\\
	&	+\frac{1}{1+q\tau}g((\nabla_{e_i}K)({e_1},{e_1}),{e_i})-\frac{p}{\alpha}g((\nabla_{e_i}K)(u,u),{e_i})+\frac{2\sqrt{\tau}q}{1+q\tau}g((\nabla_{e_i}K)({e_1},u),{e_i})\nonumber\\
	&-\frac{3q}{1+q\tau}g(K_{e_i}{e_1},u)^2+\frac{4\sqrt{\tau}pq}{\alpha(1+q\tau)}g(K_{e_i}u,u)g(K_{e_i}{e_1},u)+\frac{1}{1+q\tau}g(K_{e_1}K_{e_i}{e_1}-K_{e_i}K_{e_1}{e_1},{e_i})\nonumber\\
	&+\frac{2\sqrt{\tau}q}{1+q\tau}g(K_{e_1}K_{e_i}u-K_{e_i}K_{e_1}u,{e_i})+\frac{\tau q^2}{1+q\tau}|K_{e_i}u|^2+\frac{p}{\alpha}g(K_uu,K_{e_i}{e_i})\big\}\nonumber.
	\end{align}
According to Remark \ref{LM6'}, we obtain 
	\begin{align*}
	&g(\widetilde R(u,{e_1})\widetilde R(u,{e_1}){e_i},{e_i})=\widetilde R(u,{e_1},\widetilde R(u,{e_1}){e_i},{e_i})=R(\widetilde R(u,{e_1}){e_i},{e_i},u,e_1)\\
	&=-R({e_i},\widetilde R(u,{e_1}){e_i},u,e_1)=-\widetilde  R(u,e_1, e_i, \widetilde R(u,e_1)e_i)=-|\widetilde R(u,e_1)e_i|^2.
	\end{align*}
	Similarly, it follows
	\begin{align*}
	&g(\widetilde R(u,u)\widetilde R(u,{e_1}){e_i},{e_i})=-\sqrt{\tau}|\widetilde R(u,e_1)e_i|^2,\\
	&g(\widetilde R(u,{e_1})\widetilde R(u,u){e_i},{e_i})=-\sqrt{\tau}|\widetilde R(u,e_1)e_i|^2,\\
	&g(\widetilde R(u,u)\widetilde R(u,u){e_i},{e_i})=-\tau\widetilde R(u,e_1)e_i|^2.
	\end{align*}
	Applying the above equations in (\ref{c1}), it follows 
	that (\ref{x2}) holds.
For prove (\ref{x3}), we set $X^H=E_i=e_i^H$ and $Y^V=E_{m+j}=e_j^V/|e_j^V|, j=2,..,m$.
In this case,
 $X=e_i$ and $Y=\sqrt{\alpha^p}e_j$. Hence we have $g(Y,u)=0$ and $	Q_{p,q}(X^H,Y^V)=\frac{1}{\alpha^p}\{1+qg(Y,u)^2\}=\frac{1}{\alpha^p}$. So, (\ref{B1}) implies
\begin{align*}
& \mathcal{K}_{p,q}(E_i,E_{m+j})=\mathcal{K}_{p,q}(e_i^H,e_j^V/|e_j^V|)=\mathcal{K}_{p,q}({e_i}^H,\sqrt{\alpha^p}e_j^V)\\
&
={\alpha^p}\big\{-\frac{1}{4\alpha^p}g(\widetilde R(u,{e_j})\widetilde R(u,{e_j}){e_i},{e_i})+g(K_{e_i}u,u)^2\{-\frac{p(p+2)}{\alpha^2}\}
\\
&+\frac{2p}{\alpha}g(K_{e_i}u,u)g(K_{e_j}{e_j},{e_i})+g((\nabla_{e_i}K)({e_j},{e_j}),{e_i})\\
&	-\frac{p}{\alpha}g((\nabla_{e_i}K)(u,u),{e_i})-{3q}g(K_{e_i}{e_j},u)^2\\
&+g(K_{e_j}K_{e_i}{e_j}-K_{e_i}K_{e_j}{e_j},{e_i})
+\frac{p}{\alpha}g(K_uu,K_{e_i}{e_i})\big\},
\end{align*}
which gives us the assertion.
Considering $X^V=E_{m+1}=e_1^V/|e_1^V|$ and $Y^V=E_{m+j}=e_j^V/|e_j^V|, j=2,..,m$,
we have $X=\sqrt{\frac{\alpha^{p}}{1+q\tau}}e_1$ and $Y=\sqrt{\alpha^{p}} e_j$. Also, $|X^V|^2=|E_{m+1}|^2=|\sqrt{\frac{\alpha^{p}}{1+q\tau}}e_1^V|^2=1$ and $|Y^V|^2=|E_{m+j}|^2=|\sqrt{\alpha^{p}} e_j^V|^2=1$. Hence   (\ref{Mo1}) gives
\begin{align*}
&\mathcal{K}_{p,q}(E_{m+1},E_{m+j})=\mathcal{K}_{p,q}(\sqrt{\frac{\alpha^{p}}{1+q\tau}}e_1^V,{\sqrt{\alpha^{p}}}e_j^V)\\
&=\frac{\alpha^{2p}(1+q\tau)}{1+q\tau(1+\alpha^p)}\big\{-\frac{p^2}{\alpha^2}|K_uu|^2+\frac{p}{\alpha}\{g(K_{e_j}{e_j},K_uu)+\frac{1}{1+q\tau}g(K_{e_1}{e_1},K_uu)\}\\
&+\frac{2pq\sqrt{\tau}}{\alpha^{p+1}}{\frac{1}{1+q\tau}}[g(K_uu,K_u{e_1})]+{\frac{1}{1+q\tau}}[|K_{e_1}{e_j}|^2- g(K_{e_j}{e_j},K_{e_1}{e_1})]\nonumber\\
&+{\frac{2q\sqrt{\tau}}{1+q\tau}}[g(K_{e_1}{e_j},K_u{e_j})]-{\frac{2q\sqrt{\tau}}{1+q\tau}}[g(K_{e_j}{e_j},K_u{e_1})]+{\frac{q^2}{1+q\tau}}\tau|K_{e_j}u|^2\nonumber\\
&
-\!\frac{(\tau \mathcal{N}\alpha+2-p)p-\mathcal{N}\alpha^2}{\alpha^{p+2}}{\frac{\alpha^{p}}{1+q\tau}}\tau-\frac{p(\mathcal{M}\tau-1)-\mathcal{M}\alpha}{\alpha^{p+1}}\}.\nonumber
\end{align*}
From the above equation, (\ref{x4}) follows.
Setting $X^V=E_{m+i}=e_i^V/|e_i^V|$ and $Y^V=E_{m+j}=e_j^V/|e_j^V|, i,j=2,..,m$, 
it follows $X=\sqrt{\alpha^{p}}e_s$ and $Y=\sqrt{\alpha^{p}}e_j$. So, we have $g(Y,u)=0=g(X,u)$ and  $|X^V|^2=1$ and $|Y^V|^2=1$.
Thus (\ref{Mo1}) implies
\begin{align*}
\mathcal{K}_{p,q}(E_{m+i},E_{m+j})=&\alpha^{2p}\big\{-\frac{p^2}{\alpha^2}|K_uu|^2+\frac{p}{\alpha}\{g(K_{e_j}{e_j},K_uu)+g(K_{e_i}{e_i},K_uu)\}\\
&+[|K_{e_i}{e_j}|^2-g(K_{e_j}{e_j},K_{e_i}{e_i})]
-\frac{p(\mathcal{M}\tau-1)-\mathcal{M}\alpha}{\alpha^{p+1}}\}.\nonumber
\end{align*}
Hence we find (\ref{x5}).
\end{proof}
\begin{proposition}
	Let $(M, g, \nabla)$ be a statistical manifold with the curvature tensor $R$. If $x\in M$ and $\{e_1,\ldots,e_m\}$ is an orthonormal basis of the tangent space $T_xM$, then
	\begin{align*}
	\sum_{i,j=1}^{m}|R(e_i,e_j)u|^2=\sum_{i,j=1}^{m}|\widetilde {R}(u,e_j)e_i|^2.
	\end{align*}
\end{proposition}
\begin{proof}
	Putting $u=\sum_{i=1}^{m}u^ie_i$ and $R(e_i,e_j)e_k=\sum_{i=1}^{m}g(R(e_i,e_j)e_k,e_s)e_s$, we have
	\begin{align*}
	\sum_{i,j=1}^{m}|R(e_i,e_j)u|^2=
	\sum_{i,j,l,k=1}^{m}u^lu^kg(R(e_i,e_j)e_k,R(e_i,e_j)e_l)=
	\sum_{i,j,l,k,s=1}^{m}u^lu^kg(R(e_i,e_j)e_k,e_s)g(R(e_i,e_j)e_l,e_s).
	\end{align*}
	The above equation and Remark \ref{LM6'} imply
	\begin{align*}
	\sum_{i,j=1}^{m}|R(e_i,e_j)u|^2=&
	\sum_{i,j,l,k,s=1}^{m}u^lu^kg(\widetilde {R}(e_k,e_s)e_i,e_j)g(\widetilde {R}(e_l,e_s)e_i,e_j)=\sum_{i,j,l,k=1}^{m}u^lu^kg(\widetilde {R}(e_k,e_j)e_i,\widetilde {R}(e_l,e_j)e_i)\\
	&=\sum_{i,j=1}^{m}|\widetilde {R}(u,e_j)e_i|^2.
	\end{align*}
\end{proof}
\begin{theorem}
	Let $(M, g, \nabla)$ be a statistical manifold with scalar curvature $S$. 
	If its tangent bundle $TM$ is endowed with the Cheeger-Gromoll type metric $g_{p,q}$, 
	then the scalar curvature $S_{p,q}$ of $(TM, g_{p,q})$ satisfies the following
		\begin{align*}
		S_{p,q}=&S+\sum_{i,j=1}^{m}\bigg\{\frac{2\alpha^p-3}{4\alpha^p}|R(e_i,e_j)u|^2-\frac{3q}{4\alpha^p}(g(R(e_i,e_j)u,u))^2+g((\nabla_{e_j}K)(e_i,e_j),e_i)\\
		&-g((\nabla_{e_i}K)(e_j,e_j),e_i)+(1-\alpha^p)(g(K_{e_i}e_i,K_{e_j}e_j)-|K_{e_i}e_j|^2)\bigg\}\\
		&+2\sum_{i=1}^{m}\bigg\{A_1g(K_{e_i}e_1,e_1)^2+(A_2-\tau {p}\alpha^{p-1}(m-1))g((\nabla_{e_i}K)(e_1,e_1),e_i)\\
		&+A_3|K_{e_i}e_1|^2+(A_4 +\tau{p}\alpha^{p-1}(m-1))g(K_{e_1}e_1,K_{e_i}e_i)+A_5|\widetilde R(u,e_1)e_i|^2\\
		&
		-(m-1){p(p+2)}\alpha^{p-2}g(K_{e_i}u,u)^2+(m-2)\alpha^{2p-1}{p}g(K_{e_i}{e_i},K_uu)\nonumber\bigg\}\\
		& +2\alpha^p\sum_{i=1, j=2}^{m}\bigg\{{2p}\alpha^{-1}g(K_{e_i}u,u)g(K_{e_j}{e_j},{e_i})
		+g((\nabla_{e_i}K)({e_j},{e_j}),{e_i})
		-{3q}g(K_{e_i}{e_j},u)^2\nonumber\\
		& +g(K_{e_j}K_{e_i}{e_j}-K_{e_i}K_{e_j}{e_j},{e_i})\bigg\}+2\sum_{j=2}^{m}\bigg(B_2		|K_{e_j}e_1|^2+B_3	
		g(K_{e_j}e_j,K_{e_1}e_1)\bigg)
		\\
		&+2(m-1)B_1|K_{e_1}e_1|^2+2(m-1)B_4-(m-1)(m-2)\alpha^{2p-2}{p^2}|K_uu|^2\\
		&
		-(m-1)(m-2)\alpha^{p-1}({p(\mathcal{M}\tau-1)-\mathcal{M}\alpha}).
		\end{align*}
		\end{theorem}
	\begin{proof}
	For a local orthonormal frame $\{E_1,\ldots, E_{2m} \}$ for $TM$ with $X_{i}^H=E_i$ and $Y_i^V=E_{m+i}$, $i=1,\ldots, m$, we have
\begin{align}\label{Mo2}
S_{p,q}\!=\!\sum_{i\neq j}^{2m} \mathcal{K}_{p,q}(E_{i},E_{j})=&\!\!\sum_{i,j=1,i\neq j}^{m} \mathcal{K}_{p,q}(E_{i},E_{j})+\!2\!\sum_{i, j=1}^{m}
\mathcal{K}_{p,q}(E_{i},E_{m+j})+\!\!\sum_{i,j=1, i\neq j}^{m}\!\!
\mathcal{K}_{p,q}(E_{m+i},E_{m+j})\\
=	&\sum_{i,j=1,i\neq j}^{m} \mathcal{K}_{p,q}(E_{i},E_{j})+2\sum_{i=1}^{m}
\mathcal{K}_{p,q}(E_{i},E_{m+1})+2\sum_{i=1, j=2}^{m}
\mathcal{K}_{p,q}(E_{i},E_{m+j})\nonumber\\
&+2\sum_{j=2}^{m}
\mathcal{K}_{p,q}(E_{m+1},E_{m+j})+\sum_{i,j=2, i\neq j}^{m}
\mathcal{K}_{p,q}(E_{m+i},E_{m+j}).\nonumber
\end{align}
Applying Lemma \ref{ML}, it follows
\begin{align*}
\sum_{i,j=1,i\neq j}^{m} \mathcal{K}_{p,q}(E_{i},E_{j})=&\sum_{i,j=1,i\neq j}^{m}\mathcal{K}(e_i,e_j)+\sum_{i,j=1}^{m}\bigg\{-\frac{3}{4\alpha^p}|R(e_i,e_j)u|^2-\frac{3q}{4\alpha^p}(g(R(e_i,e_j)u,u))^2\\
&+g((\nabla_{e_j}K)(e_i,e_j),e_i)-g((\nabla_{e_i}K)(e_j,e_j),e_i)+g(K_{e_i}e_i,K_{e_j}e_j)-|K_{e_i}e_j|^2\bigg\},
\end{align*}
\begin{align*}
2\sum_{i=1}^{m}
\mathcal{K}_{p,q}(E_{i},E_{m+1})=&2\sum_{i=1}^{m}\bigg\{A_1g(K_{e_i}e_1,e_1)^2+A_2g((\nabla_{e_i}K)(e_1,e_1),e_i)+A_3|K_{e_i}e_1|^2+A_4g(K_{e_1}e_1,K_{e_i}e_i)\\
&+A_5|\widetilde R(u,e_1)e_i|^2,\nonumber\bigg\},
\end{align*}

	\begin{align*}
& 	2\sum_{i=1, j=2}^{m}\mathcal{K}_{p,q}(E_i,E_{m+j})=-\frac{1}{2}\sum_{i=1, j=2}^{m}g(\widetilde R(u,{e_j})\widetilde R(u,{e_j}){e_i},{e_i})
-2\sum_{i=1 }^{m}(m-1){p(p+2)}\alpha^{p-2}g(K_{e_i}u,u)^2\\
&\ \ \ \ \ \ \ \ \ \ \ \ \ \ \ \ \ \ \ \ \ \ \ \ \ \ \ \ \ \ \ \ \ \  +{4p}\alpha^{p-1}\sum_{i=1, j=2}^{m}g(K_{e_i}u,u)g(K_{e_j}{e_j},{e_i})
+2\alpha^p\sum_{i=1, j=2}^{m}g((\nabla_{e_i}K)({e_j},{e_j}),{e_i})\nonumber\\
&\ \ \ \ \ \ \ \ \ \ \ \ \ \ \ \ \ \ \ \ \  \ \ \ \ \ \ \ \  \  \ \ \ \ \ -2{p}\alpha^{p-1}\sum_{i=1 }^{m}(m-1)g((\nabla_{e_i}K)(u,u),{e_i})
-{6q}\alpha^p\sum_{i=1, j=2}^{m}g(K_{e_i}{e_j},u)^2\nonumber\\
&\ \ \ \ \ \ \ \ \ \ \ \ \ \ \ \ \ \ \ \ \  \ \ \ \ \  \ \ \ \ \  \ \ \ \ +2\alpha^p\sum_{i=1, j=2}^{m}g(K_{e_j}K_{e_i}{e_j}-K_{e_i}K_{e_j}{e_j},{e_i})
+2{p}\alpha^{p-1}\sum_{i=1 }^{m}(m-1)g(K_uu,K_{e_i}{e_i}),\nonumber
\end{align*}
	\begin{align*}
\	2\sum_{j=2}^{m}\mathcal{K}_{p,q}(E_{m+1},E_{m+j})=&\	2(m-1)B_1|K_{e_1}e_1|^2+2B_2	\sum_{j=2}^{m}	|K_{e_j}e_1|^2+2B_3	\sum_{j=2}^{m}
g(K_{e_j}e_j,K_{e_1}e_1)+2(m-1)B_4,
\end{align*}
and
	\begin{align*}
\sum_{i,j=2, i\neq j}^{m}\mathcal{K}_{p,q}(E_{m+i},E_{m+j})=&-(m-1)(m-2)\alpha^{2p-2}{p^2}|K_uu|^2+(m-2)\alpha^{2p-1}{p}\{\sum_{j=2}^{m}g(K_{e_j}{e_j},K_uu)\\
&+\sum_{i=2}^{m}g(K_{e_i}{e_i},K_uu)\}
 +\alpha^{2p}\sum_{i,j=2, i\neq j}^{m}|K_{e_i}{e_j}|^2-\alpha^{2p}\sum_{i,j=2, i\neq j}^{m}g(K_{e_j}{e_j},K_{e_i}{e_i})\\
 &
-(m-1)(m-2)\alpha^{p-1}({p(\mathcal{M}\tau-1)-\mathcal{M}\alpha}).
\end{align*}
Using the above equations in (\ref{Mo2}), we get the assertion.

\end{proof}

\section*{Conclusion}
In this paper, we investigated the geometry of the tangent bundle $TM$ of a statistical manifold $(M,g,\nabla)$ equipped with a two-parameter family of generalized Cheeger--Gromoll metrics $g_{p,q}$. We derived explicit formulas for the Levi--Civita connection $\nabla^{p,q}$ and expressed the curvature of $(TM, g_{p,q})$ in terms of the Riemannian curvature and the skewness tensor $K$ of the base manifold. The sectional curvature $\mathcal{K}_{p,q}$ and scalar curvature $S_{p,q}$.
Moreover, we analyzed the behavior of geodesics, established conditions for totally geodesic fibers, and determined when the geodesic flow is incompressible. Necessary and sufficient conditions for $(TM, g_{p,q})$ to have constant sectional curvature were also identified. Several illustrative examples, including statistically deformed Euclidean spaces and the manifold of normal distributions, demonstrated the interplay between statistical structures and lifted geometric constructions.

These results open multiple directions for future research. Potential avenues include investigating curvature bounds, Ricci solitons, and geometric flows on tangent bundles of statistical manifolds. The framework can be extended to more general information-geometric models, offering deeper insights into the differential-geometric properties of statistical and probabilistic structures. Applications in optimization, machine learning, and information theory are also promising, as understanding the geometry of tangent bundles can lead to new theoretical and practical developments.\\\\

{\bf Conflict of Interest}\\
The authors declare that they have no conflict of interest.



\begin{thebibliography}{99}
	\bibitem{Abbassi}	M. Abbassi and M. Sarih, {\it Killing Vector Fields on Tangent Bundles with
	Cheeger Gromoll Metric}, Tsukuba J. Math., {\bf 27}(2) (2003), 295-306.
	
	\bibitem{AM} S. Amari, {\it Information geometry of the EM and em algorithms for
	neural networks}, Neural Networks {\bf 8}(9) (1995), 1379--1408.
	
	\bibitem{A} S. Amari, {\it Natural Gradient Works Efficiently in Learning}, Neural Computation, {\bf 10}(2) (1998), 251--276.
	
	\bibitem{AMH}  S. Amari and H. Nagaoka, {\it Methods of information geometry},
	American Mathematical Society, 2000.
	
	\bibitem{B} M. Barbosu, {\it Sasaki metric and the phase space}, Carpathian Journal of Mathematics, {\bf 14} (1998), 139--142.
	
	\bibitem{BG} L. P. Bettmann and J. Goold, {\it Information geometry approach to quantum stochastic thermodynamics}, Phys. Rev. E, {\bf 111} (2025), 014133. DOI: https://doi.org/10.1103/PhysRevE.111.014133
	
	\bibitem{BN}  M. Belkin,  P. Niyogi and V. Sindhwani, {\it Manifold regularization:
	a geometric framework for learning from labeled and unlabeled examples},
	Journal of Machine Learning Research, {\bf 7} (2006), 2399-2434.
	
		\bibitem{BL} M. Benyounes, E. Loubeau and C.M. Wood, {\it Harmonic sections of Riemannian
	vector bundles, and metrics of Cheeger-Gromoll type}, Diff. Geom. Appl., {\bf 25} (2007),  322--334.
	
	\bibitem{Blaga} A.M. Blaga, and C. Crasmareanu,  {\it Statistical Structures in Almost Paracontact Geometry}, Bulletin of the Iranian Mathematical Society {\bf 44}  (2018), 1407-1413.
		\bibitem{CU} O. Calin and C. Udri\c{s}te, {\textit{Geometric Modeling in Probability and Statistics}}, Springer, Cham, Switzerland 2014.
	\bibitem{Cai} D. Cai, X. Liu, and L. Zhang,  {\it Inequalities on generalized normalized $\delta$-Casorati curvatures for submanifolds in statistical manifolds of constant curvatures}, (Chinese) J. Jilin Univ. Sci. {\bf 57}(2)  (2019), 206-212.
	
	
	
	\bibitem{CG1}   J. Cheeger and D. Gromoll,  \textit{ On the structure of complete manifolds of nonnegative curvature}, Ann. Math. {\bf96} (1972), 413-443.
	
	\bibitem{D} R. Di Sipio, {\it Rethinking LLM training through information geometry and quantum metrics}, 	arXiv:2506.15830.
	
	\bibitem{F} R. A. Fisher, {\it On the mathematical foundations of theoretical statistics},  Phil. Trans. Roy. Soc. London, {\bf 222} (1922), 309--368.
	\bibitem{Furuhata} H. Furuhata, I. Hasegawa, Y. Okuyama, K. Sato and M.H. Shahid, {\it Sasakian
		statistical manifolds}, J. Geom. Phys.  {\bf 117} (2017), 179-–186. 
	
	
	\bibitem{Opozda}
	B. Opozda,  {\it On the tangent bundles of statistical manifolds}, In Proceedings of the Geometric Science of Information GSI 2023,
	St. Malo, France, 30 August–1 September 2023.
	\bibitem{PNI}	E. Peyghan, L. Nourmohammadifar and D. Iosifidis,  {\it Statistical conformal Killing
		vector fields for FLRW spacetime}, Phys Scr. {\bf 99}, 065253, (2024).
	\bibitem{Peyghan}  E. Peyghan, L. Nourmohammadifar and S. Uddin,  {\it Musical Isomorphisms and Statistical
		Manifolds}, Mediterr. J. Math. (2022), Doi.org/10.1007/s00009-022-02141-z
	\bibitem{PNII}	E. Peyghan,  L. Nourmohammadifar, and I. Mihai, {\it Statistical Submanifolds Equipped with
		${\mathcal{F}}$-Statistical Connections}, Mathematics,  {\bf 12}(16) (2024), 2492. https://doi.org/10.3390/math12162492
	\bibitem{PNII3}	 E. Peyghan, F. L. Nourmohammadi, A. Tayebi, {\it Cheeger–Gromoll type metrics on the (1,1)-tensor bundles}, J.
	Contemp. Math. Anal., {\bf 48}(6),  (2013), 59-70.
	\bibitem{OP}	B. Opozda,  {\textit Bochner’s technique for statistical structures}, Ann. Global Anal. Geom. {\bf 48} (2015), 357--395.
	
	\bibitem{S} S. Sasaki, {\it On the differential geometry of tangent bundles of Riemannian manifolds I}, Tohoku Mathematical Journal, {\bf 10}(3) (1958), 338--354.

		
	
	




\end{thebibliography}
\end{document}